\documentclass[12pt]{amsart}

\usepackage{amsmath, amssymb, amsthm, mathrsfs, newtxtext}
\usepackage[bottom=3.5cm, top=3cm, left=3cm, right=3cm]{geometry}
\makeatletter
\@namedef{subjclassname@2020}{
\textup{2020} Mathematics Subject Classification}
\makeatother


\newtheorem{thm}{Theorem}
\newtheorem{prop}{Proposition}[section]
\newtheorem{lem}{Lemma}[section]
\newtheorem{defn}{Definition}
\newtheorem{rmk}{Remark}
\newtheorem{coro}{Corollary}

\newcommand{\card}{\mathrm{Card}}
\newcommand{\dist}{\mathrm{d}}
\newcommand{\vol}{\mathrm{vol}}
\newcommand{\m}{\mathrm{m}}	
\newcommand{\rv}{\mathrm{v}}
\newcommand{\rw}{\mathrm{w}}

\newcommand{\ld}{\ell}		
\newcommand{\hs}{h}		
\newcommand{\htop}{h_{\mathrm{top}}}
\newcommand{\gs}{g_{\T}}

\newcommand{\hh}{\mathrm{m}^{\Q}}		
\newcommand{\Hh}{\mathrm{m}^{\widetilde\Q}}

\newcommand{\g}{\mathbf{g}}		
\newcommand{\G}{\mathrm{G}}		

\newcommand{\kk}{\mathrm{k}}		
\newcommand{\bb}{\mathrm{b}}	

\newcommand{\HH}{\mathcal{H}} 
\newcommand{\M}{\mathcal{M}}  	
\newcommand{\MM}{\widetilde{\mathcal{M}}}
\newcommand{\N}{\mathcal{N}}   
\newcommand{\LL}{\mathcal{L}}   
\newcommand{\PP}{\mathcal{P}}
\newcommand{\T}{\mathcal{T}}    	
\newcommand{\W}{\mathcal{W}}  	
\newcommand{\Q}{\mathcal{Q}}   	
\newcommand{\C}{\mathcal{C}}   	
\newcommand{\K}{F^{\mathrm{BM}}}  	 	
\newcommand{\Fi}{\mathscr{F}}	

\newcommand{\J}{F^{su}}

\begin{document}
\title[Central limit theorem of Brownian motions]{Central limit theorem of 
Brownian motions 
in pinched negative curvature}

\author{Jaelin Kim}

\address{
Department of Mathematical Sciences, Seoul National University, Seoul 151-747, Republic of Korea.
}
\email{kimjl@snu.ac.kr}

\begin{abstract}
We prove the central limit theorem of random variables induced by distances to Brownian paths and Green functions on the universal cover of Riemannian manifolds of finite volume with pinched negative curvature.
We further provide some ergodic properties of Brownian motions and an application of the central limit theorem to the dynamics of geodesic flows in pinched negative curvature.
\end{abstract}

\maketitle

\section{Introduction}
Let $\MM$ be a simply connected complete Riemannian manifold of dimension $d\ge2$ with pinched negative curvature; its sectional curvature is uniformly bounded between two negatives.
We further assume that $\MM$ admits a finite-volume quotient $\M$ and the first derivative of the sectional curvature is uniformly bounded.

The Brownian motion $(\widetilde\omega_{t})_{t\in\mathbb{R}_{+}}$ on $\MM$ starting from $x$ is transient as $\MM$ is negatively curved. 
Therefore, the distance $\dist(x, \widetilde\omega_{t})$ goes to infinity as $t \to \infty$ with probability 1
and its asymptotic growth is linear (\cite{Guivarch_1981}): there is $\ld>0$ such that
\begin{equation*}
\ld
=
\lim_{t\to\infty} \frac{1}{t}\dist_{}(x, \widetilde\omega_{t}).
\end{equation*}
Due to the pinched negative curvature, the Green function $\G(x,y)$ on $\MM$ tends to zero as $\dist(x,y) \to \infty$. 
Hence $\G(x, \widetilde\omega_{t}) \to 0$ as $t\to\infty$ and it decays exponentially fast with probability 1 (\cite{Kai_1986}): there exists $\hs>0$ such that
\begin{equation*}
\hs 
=
\lim_{t\to\infty} -\frac{1}{t}\log\G(x, \widetilde\omega_{t}).
\end{equation*}

Even though Brownian motions on manifolds with pinched negative curvature has been studied for a long time,
the majority of the results holds for either every Cartan-Hadamard manifolds or co-compact ones
and few are known for the cases in between, especially for the co-finite manifolds $\MM$.
Our main result, the central limit theorem of random processes
$Y^{\ld}_{t}(\widetilde\omega)= \dist_{}(x,\widetilde\omega_{t})-t\ld$
and
$Y^{\hs}_{t}(\widetilde\omega)= \log \G(x, \widetilde\omega_{t})+t\hs$,
is a generalization of the central limit theorem in co-compact manifolds proved by F. Ledrappier in \cite{Led_1995}.

\begin{thm} \label{CLT}
The distributions of 
$\frac{1}{\sigma_{\bb}\sqrt{t}}Y^{\ld}_{t}$ and 
$\frac{1}{\sigma_{\kk}\sqrt{t}}Y^{\hs}_{t}$ are asymptotically normal for some positive constants $\sigma_{\bb}, \sigma_{\kk}$. 
More precisely, for every $x \in \MM$,
\begin{align*}
\mathbb{P}_{x}
\left[
\frac{Y^{\ld}_{t}}{\sigma_{\bb}\sqrt{t}}
\le r
\right] ,
\mathbb{P}_{x}
\left[
\frac{Y^{\hs}_{t}}{\sigma_{\kk}\sqrt{t}}
\le r
\right] 
\to 
\frac{1}{\sqrt{2\pi}}\int_{-\infty}^{r} \exp\left( - \frac{s^{2}}{2} \right) ds,
\textrm{ as } t\to\infty,
\end{align*}
where $\mathbb{P}_{x}$ is the probability measures on the space 
$\C(\mathbb{R}_{+}, \MM)$ 
of continuous sample paths which defines the Brownian motion on $\MM$ starting from $x$.
\end{thm}

F. Ledrappier introduced a double process to provide a lower bound for the expectation of the Gromov product at Brownian points in \cite{Led_1995}. 
The lower bound implies the contraction property of the foliated Brownian motion, which plays an important role in the proof of the central limit theorem.
However, since the double process argument is not valid in the absence of compactness, 
we instead provide an argument using the $\C^{2}$-convergence of the normalized distance functions to the Busemann function in pinched negatively curved manifolds.
Although the resulting lower bound is less sharp than the lower bound by the double process argument, it is sufficient for the proof of the contraction property.

As in \cite{Led_1995}, we use the contraction property of the foliated Brownian motion (Theorem \ref{Ctr}) on H\"older spaces to solve the leafwise heat equation on the unit tangent bundle for the foliated Laplacian. 
We construct Martingales from the solutions of the heat equation with the initial conditions of the Busemann function and the logarithm of the Martin kernel of the Brownian motion. 
We prove that they are asymptotically normal and have the same distributions with the random variables of our interest.

As a consequence of the central limit theorem, we provide a characterization for the asymptotic harmonicity of $\MM$ with an assumption for thermodynamic formalism.
We say that $\MM$ is \textit{asymptotically harmonic} if the mean curvature of the horospheres of $\MM$ is constant. 
If $\MM$ is asymptotically harmonic then the Liouville measure on the unit tangent bundle of $\M$ has maximal entropy for the geodesic flow.
The characterization reveals an interplay between the stochastic properties, the geometry and the dynamics of the geodesic flow of $\MM$.
Indeed, an asymptotically harmonic manifold $\MM$ is a symmetric space if it is the universal cover of a compact negatively curved manifold (\cite{FouLab_1992}, \cite{BFL_1992}, \cite{Led_1990}).
The Martin kernel of the Brownian motion gives rise to a H\"older continuous function $\K$ on $\T^{1}\M$, 
which helps us understand the asymptotic behavior of Brownian paths and correlation with geodesics.
An equilibrium state of $\K$ is a geodesic flow-invariant Borel probability measure on $\T^{1}\M$ which maximizes the pressure of $\K$.
For compact manifolds, 
every H\"older continuous function admits a unique equilibrium states (\cite{Franco_1977})
while the existence is not always guaranteed for finite-volume manifolds.

\begin{thm}\label{AsympH}
If $\K$ admits an equilibrium state, then
\begin{equation*}
\sigma_{\kk}^{2}\ge 2\hs.
\end{equation*}
The equality holds if and only if $\MM$ is asymptotically harmonic.
\end{thm}

In Section 2, we introduce the heat kernel and the Brownian motion on $\MM$. We also recall preliminaries of the geometry of manifolds with pinched negative curvature, the ergodic theory and thermodynamic formalisms for their geodesic flow.
We prove Theorem \ref{CLT} in Section 3 while Section 4 is devoted to the proof of the contraction property (Theorem \ref{Ctr}).
Section 4 also contains a diagonal estimate of the heat kernel and the proof of exponential ergodicity of the Brownian motion on $\M$.
In Section 5, we prove ergodic properties of the Brownian motions which generalize the results in \cite{Led_1988}. 
We conclude the section with the proof of Theorem \ref{AsympH}.
\\

\textit{Acknowledgement}. It is a pleasure to thank Fran\c{c}ois Ledrappier and Seonhee Lim for sharing their insights and helpful comments. The work is supported by Samsung Science and Technology Foundation under Project Number SSTF-BA1601-03.

\section{Preliminaries} \label{Pre}

Let $(\M, g)$ be a complete finite-volume Riemannian manifold of dimension $d\ge2$. 
We say that $\M$ has \textit{pinched negative curvature} if 
\begin{equation*}
-b^{2}\le \mathrm{sec}_{\M} \le -a^2
\end{equation*}
for some positive numbers $b>a>0$.

We assume that $\M$ has pinched negative curvature and $|\nabla \sec_{\M}| \le c$ for some $c>0$.
Let $\MM\to\M$ be the universal cover with the group of deck transformation $\Gamma$ acting isometrically on $\MM$. We also denote the lift of the metric on $\M$ to $\MM$ by $g$.
Let $\dist_{}$ be the Riemannian distance of $\MM$ and $\vol:=\vol_{\MM}$ the Riemannian volume on $\MM$.

A number of examples can be constructed from noncompact finite-volume hyperbolic manifolds by perturbing the metric near cusps.
See \cite{DPPS_2009}, \cite{DPPS_2017} for the detail.

\subsection{Geometry of pinched negative curvature}
Since $\MM$ has pinched negative curvature, 
the metric space
$(\MM, \dist_{})$ 
is a CAT(0)-space. Hence we consider its boundary at infinity $\partial\MM$, also called the \textit{visual boundary}.
Fix $x\in\MM$. 
A sequence $(z_{n})$ in $\MM$ converges to a point $\xi$ in $\partial\MM$ 
if and only if
$z_{n}\to\infty$ 
and 
the sequence of normalized distance functions
\begin{equation*}
f_{n}(y)=\bb(y, x, z_{n}):=\dist_{}(y, z_{n})-\dist_{}(x, z_{n})
\end{equation*}
converges uniformly on compact sets in $\C(\MM)$. 
We denote the limit function by 
$\bb(y, x, \xi)$, which we call the \textit{Busemann function} based at $\xi$.
The convergence of $z_{n}$ to $\xi$ is independent of the choice of $x$.
An important remark is that $f_{n}$ converges to the Busemann function $\C^{2}$-uniformly on compact sets:

\begin{prop} \label{CCB}
(\cite{Bal_1995}) 
Let $f_{n}(y)= \dist_{}(y, z_{n})-\dist_{}(x, z_{n})$ and $z_{n}\to \xi \in \partial\MM$. Then
\begin{align*}
\nabla f_{n} 
&\to 
\nabla \bb(\cdot, x, \xi),
\\
\nabla_{\rv}\nabla f_{n} 
&\to 
\nabla_{\rv} \nabla \bb(\cdot, x, \xi),
\end{align*}
uniformly on compact sets. 
$\nabla \bb(\cdot, x, \xi)$ means the covariance derivative of $y\mapsto\bb(y,x,\xi)$.
\end{prop}

Let $\Delta= \mathrm{div} \nabla$ be the Laplace-Beltrami operator on 
$(\MM,g)$.
If $\{e_{1}, \dots, e_{d}\}$ is an orthonormal frame on an open set $U$, 
for each $\C^{2}$-function $f$ on $U$,
\begin{equation}\label{Localex}
\Delta f = \sum_{j=1}^{d}\langle e_{j}, \nabla_{e_{j}}\nabla f\rangle_{g}
\end{equation} 
on $U$. Applying Proposition \ref{CCB} to each summand of (\ref{Localex}), we obtain the following result.

\begin{prop} \label{C2conv}
Let $f_{n}(y)= \dist_{}(y, z_{n})-\dist_{}(x, z_{n})$ and $z_{n}\to \xi \in \partial\MM$. Then
$\Delta f_{n}$ converges to $\Delta \bb(\cdot, x, \xi)$ uniformly on compact sets.
\end{prop}

The visual boundary $\partial\MM$ is equipped with a distance.
For $\xi, \eta \in\partial \MM$ with $z_{n}, w_{n}\in\MM$ which converge to $\xi, \eta$ respectively, we define the Gromov product of $\xi$ and $\eta$ at $x\in \MM$ by
\begin{equation*}
(\xi|\eta)_{x}:=\lim_{n\to\infty} \dist(x, z_{n}) + \dist(x, w_{n}) -\dist(z_{n}, w_{n}).
\end{equation*}
Then for $\tau>0$ small enough, $d_{\infty}^{x,\tau}(\xi, \eta):=\exp[-\tau(\xi|\eta)_{x}]$ is a distance function on the visual boundary $\partial\MM$ (see \cite{BriHae}).

Let $\pi:\T\MM\to\MM$ be the tangent bundle of $\MM$. 
We endow $\T\MM$ with a Riemannian metric $\gs$ called the \textit{Sasaki metric}, induced by the Riemannian structure $g$ of $\MM$ and its Levi-Civita connection $\nabla$.
We consider the unit tangent bundle
$\T^{1}\MM=\{ \rv\in \T\MM: \|\mathrm{v}\|^{2}=\langle \mathrm{v}, \mathrm{v}\rangle_{g}=1\}$
of $\MM$, which is a submanifold of $\T\MM$ and also a sphere bundle of $\MM$.
We denote the geodesic flow on $\T^{1}\MM$ by
$\g^{t}:\T^{1}\MM \to \T^{1}\MM$. 
We also denote by $\g^{t}$ the geodesic flow on the unit tangent bundle $\T^{1}\M$ of $\M$.

We introduce the stable foliation $\widetilde\W^{s}$ and the strong unstable foliation $\widetilde\W^{su}$ of $\T^{1}\MM$ which will play an important role in the following sections. Their leaves are defined by
\begin{align*}
\widetilde\W^{s}(\rv)
&=
\left\{
\rw\in\T^{1}\MM: 
\lim_{t \to \infty}
\dist_{}(\gamma_{\rv}(t+s), \gamma_{\rw}(t))
=0,
\, 
\exists s
\right\},
\\
\widetilde\W^{su}(\rv)
&=
\left\{
\rw\in\T^{1}\MM:
\lim_{t\to\infty}
\dist_{}(\gamma_{\rv}(t), \gamma_{\rw}(t))
=0
\right\},
\end{align*}
Where $\gamma_{\rv}$ is the geodesic generated by $\rv$.
Note that $\widetilde\W^{su}$ consists of unit normal bundles of level sets of Busemann functions and leaves are transversal to the stable foliation with angle uniformly bounded away from zero (Lemma 7.4. in \cite{PPS}).

The \textit{stable distribution} $\widetilde{E}^{s}$ of $\T^{1}\MM$ is a rank $d$-subbudle of the tangent bundle $\T\T^{1}\MM \to \T^{1}\MM$ of $\T^{1}\MM$ whose fibers are tangent spaces of stable leaves: 
$\widetilde{E}^{s}_{\rv}:= \T_{\rv}\widetilde\W^{s}(\rv)$. 
Since $\widetilde\W^{s}(\rv)$ is diffeomorphic to $\MM$ via $\pi:\T^{1}\MM\to\MM$ for $\rv \in \T^{1}\MM$, 
we endow stable leaves of $\widetilde\W^{s}$ with a metric $g_{s}$ induced from the metric $g$ on $\MM$: 
for $\rv\in\T^{1}_{x}\MM$, define $g_{s}$ on $\widetilde{E}^{s}_{\rv}=\T_{\rv}\widetilde\W^{s}(\rv)$ from $g$ on $\T_{x}\MM$.

For each point $x\in\MM$ and point at infinity $\xi\in\partial\MM$,
there is a unique unit vector $\rv$ in $\T^{1}\MM$ such that
$\gamma_{\mathrm{v}}(t)$ converges to $\xi$ at $t\to \infty$.
Conversely, for every geodesic $\gamma$, $\gamma(t)$ converges to a point $\xi$ in $\partial\MM$.
We denote the limit point $\xi$ of $\gamma_{\rv}(t)$ by $\rv_{+}$.
This gives a useful identification of $\T^{1}\MM$ with $\MM\times\partial\MM$.
With such identification, we have that for $\rv=(x,\xi)$, $\widetilde\W^{s}(\rv)=\MM\times\{\xi\}$.
Moreover, 
$\nabla_{y} \bb(y, x, \xi)= (y, \xi)$.

Let $X: \T^{1}\MM \to \widetilde{E}^{s}$ be a section of the stable distribution which is leafwise $\C^{1}$, 
i.e., the restriction $X|_{\widetilde\W^{s}(x, \xi)}$ is $\C^{1}$ on $\widetilde\W^{s}(x, \xi)$ for each $(x, \xi)\in \T^{1}\MM$.
We identify $X|_{\widetilde\W^{s}(x, \xi)}$ with a $\C^{1}$-vector field $X^{\xi}$ on $\MM$ for each $\xi$.
We define the $g_{s}$-divergence $\mathrm{div}_{s}$ by
\begin{equation*}
\mathrm{div}_{s}X(x, \xi)= \mathrm{div} X^{\xi}(x).
\end{equation*}
Let $u \in \C(\T^{1}\MM)$ be a leafwise $\C^{2}$-function; $u|_{\widetilde\W^{s}(\rv)}$ is $\C^{2}$ on $\widetilde\W^{s}(\rv)$. 
Thus for each $\xi\in\partial\MM$, $u^{\xi}(x):= u(x, \xi)$ is $\C^{2}$ on $\MM$.
We define the \textit{foliated Laplacian} $\Delta_{s}$ by 
\begin{equation*}
\Delta_{s} u= \mathrm{div}_{s} \nabla u,
\end{equation*}
where $\nabla u (x, \xi) := \nabla u^{\xi} (x)$.

\subsection{Brownian motions}
The heat kernel 
$\wp:(0,\infty)\times\MM\times\MM\to(0,\infty)$ 
is the fundamental solution of the heat equation: 
\begin{align*}
\partial_{t}\wp(t,x,y)
&=
\Delta_{y}\wp(t, x,y),\\
\lim_{t \downarrow 0} \wp(t, x, y) 
&= 
\delta_{x}(y).
\end{align*}
The limit in the last equation means that for each $f\in\C_{b}(\MM)$,
\begin{equation*}
\lim_{t\downarrow0}
\int_{\MM}
\wp(t, x, y)f(y)
d\vol_{\MM}(y)
=f(x).
\end{equation*}
Since the curvature of $\MM$ is negatively pinched, $\Delta$ is (weakly) coercive, i.e., 
the Green function of $\Delta$
$$\G(x,y):=\int_{0}^{\infty}\wp(t, x, y)dt$$ 
is finite for $x \ne y\in \MM$.

For  $\kappa<0$, if $\wp_{\mathbb{H}^{d}(\kappa)}(t, x,y)$ is the heat kernel on the $d$-dimensional hyperbolic space $\mathbb{H}^{d}(\kappa)$ of constant curvature $\kappa$, $\wp_{\mathbb{H}^d(\kappa)} (t, x, y)$ depends only on $t$ and $\dist_{\mathbb{H}^{d}(\kappa)}(x,y)$. 
The following comparison theorem of the heat kernel is also due to the pinched negative curvature.

\begin{prop} \label{HCom}
(Heat kernel comparison theorem, \cite{Hsu}) 
\begin{align*}
\wp_{\mathbb{H}^{d}(-b^{2})}(t, \dist(x,y))
\le
\wp(t,x,y)
\le
\wp_{\mathbb{H}^{d}(-a^{2})}(t, \dist(x,y)).
\end{align*}
\end{prop}

Note that $\wp(t,x,y)$ determines a unique family of probability measures on the space 
$\widetilde\Omega= \C(\mathbb{R}_{+}, \MM)$ of sample paths.
For each $x\in\MM$, 
we define the probability measure 
$\mathbb{P}_{x}$ 
on the cylinder sets in 
$\widetilde\Omega$
by
\begin{align*}
&\mathbb{P}_{x}
\left[ 
\widetilde\omega_{t_{i}}\in A_{i}, t_{1}<\cdots<t_{k} 
\right]
=
\\&
\int_{A_{k}}
\cdots
\int_{A_{1}} 
\wp(t_{1}, x, y_{1})
\wp(t_{2}-t_{1}, y_{1}, y_{2})
\times
\cdots 
\times
\wp({t_{k}-t_{k-1}}, y_{k-1}, y_{k}) 
d \vol(y_{1})
\cdots
d \vol(y_{k}).
\end{align*}
By Kolmogorov extension theorem, $\mathbb{P}_{x}$ extends to a unique probability measure on $\widetilde\Omega$. 
For $s\ge0$, we denote the projection map $\widetilde\omega \mapsto \widetilde\omega_{s}$ by $\pi_{s}:\widetilde\Omega\to\MM$.
Let $\Fi_{t}=\Fi_{t}(\MM):=\sigma\{\pi_{s}\}_{0\le s \le t}$ be the smallest $\sigma$-algebra for which the projections $\pi_{s}$ are measurable.
The canonical process $\widetilde{Z}_{t}(\widetilde\omega):= \widetilde\omega_{t}$ of the filtered space $( \widetilde\Omega, \{ \Fi_{t} \}_{0\le t\le\infty} )$ forms a Markov process with respect to $\mathbb{P}_{x}$, 
which is called the \textit{Brownian motion} on $\MM$ with initial distribution $\delta_{x}$, 
for each $x\in\MM$.

Let $\Omega=\mathcal{C}(\mathbb{R}_{+},\M)$.
For each $x\in\M$ and its lift $\widetilde{x}\in\MM$,  
we also denote the push-forward measure of $\mathbb{P}_{\widetilde{x}}$ by $\mathbb{P}_{x}$.
Then the canonical process $Z_{t}$ of $(\Omega, (\Fi_{t}(\M))_{0\le t\le\infty}, (\mathbb{P}_{x})_{x\in\M})$ is a Markov process, 
which we call the Brownian motion on $\M$. 
This process is the projected process of the Brownian motion on $\MM$. 
The stationary measure of the Brownian motion is the probability measure which defines the Brownian motion with initial distribution $m$:
$\mathbb{P}_{m}=\int_{\M} \mathbb{P}_{x}\,d\m(x)$ 
where $\m$ is the normalized Riemannian volume on $\M$. 
The shift dynamical system on the path space $(\Omega, \mathscr{S}^{t}, \mathbb{P}_{m})$ is ergodic since $\M$ is connected, 
where $\mathscr{S}^{t}\omega_s =\omega_{t+s}$ for $\omega\in\Omega$.

Let $r(\omega, t)= \dist_{}(\widetilde\omega_{0}, \widetilde\omega_{t})$ where $\widetilde\omega$ is a lift of $\omega$. 
Then since $r$ is a sub-additive cocycle, that is, $r(\omega,{t+s})\le r(\omega,t)+r(\mathscr{S}^{t}\omega, s)$ for every $s,t>0$, 
there exists a positive constant $\ld$, which is called \textit{the linear drift} of the Brownian motion, such that for every $x\in\MM$ and for a.s. $\omega\in\Omega$
\begin{equation*}
\ld
=\lim_{t\to \infty}\frac{1}{t}r(\omega, t)
=\lim_{t\to\infty}\frac{1}{t}\dist_{}(x, \widetilde\omega_{t})
\end{equation*}
due to the subadditive ergodic theorem (\cite{Kingman_1968}). If $\M$ has constant negative curvature $-a^2$, then $\ld = (d-1)a$.

For a fixed $x\in\MM$, the exponential map at $x$ induces a polar coordinate on $\MM\setminus\{x\}$:
\begin{align*}
(0,\infty)\times \T^{1}_{x}\MM
&\to
\MM\setminus\{x\}\\
(r, \rv)
&\mapsto
\exp_{x} r\rv.
\end{align*}
Note that $\T^{1}_{x}\MM$ inherits the Riemannian metric $g_{\mathbb{S}}$ of the unit sphere $\mathbb{S}^{d-1}$ from $(\MM, g)$ and write $g$ as
\begin{equation*}
g= dr^{2}+ \lambda_{x}(r, \rv) g_{\mathbb{S}},
\end{equation*}
for some smooth function $\lambda_{x}$ on 
$\MM\setminus\{x\}=(0,\infty)\times\T^{1}_{x}\MM$.

For $\widetilde\omega\in \widetilde\Omega$, 
we write $r(\widetilde\omega, t)= \dist_{}(\widetilde\omega_{0}, \widetilde\omega_{t})$
and
let
$\theta(\widetilde\omega,{t})$
be the unit vector in 
$\T^{1}_{\widetilde\omega_{0}}\MM$ 
with
$\exp_{\widetilde\omega_{0}} 
\left[
r(\widetilde\omega,t)\theta(\widetilde\omega,t) 
\right] 
= \widetilde\omega_{t}$. 
\begin{prop} (\cite{Prat_1975}, \cite{Pin_1978})
For every $x\in\MM$ and $\mathbb{P}_{x}$-a.e. $\widetilde\omega$,
the limit $\lim_{t\to\infty} \theta(\widetilde\omega, t)$ exists.
\end{prop}

Since $r(\widetilde\omega, t)\to \infty$ as $t\to\infty$ for $\mathbb{P}_{x}$-a.e. $\widetilde\omega$, the limit
$\widetilde\omega_{\infty}
:=
\lim_{t\to\infty} \widetilde\omega_{t}$
exists for $\mathbb{P}_{x}$-a.e. $\widetilde\omega$. 
In addition, the Brownian path roughly follows the geodesic $\gamma_{\theta(\widetilde\omega, \infty)}$ (\cite{Led_1988}):
\begin{equation}\label{roughpath}
\lim_{t\to\infty} 
\frac{1}{t}
\dist 
\left( 
\widetilde\omega_{t}, 
\exp_{x} \left[ r(\widetilde\omega, t)\theta(\widetilde\omega, \infty) \right] 
\right)
=0.
\end{equation}
We can replace $r(\widetilde\omega, t)$ by $\ld t$. 
We denote the asymptotic distribution of Brownian paths starting from $x$ by $\nu_{x}$, i.e., 
\begin{equation*}
\nu_{x}(U):= \mathbb{P}_{x}\left[ \widetilde\omega: \widetilde\omega_{\infty}\in U \right],
\textrm{ for }U\subset \partial\MM.
\end{equation*}
Since the family $(\mathbb{P}_{x})$ is $\Gamma$-equivariant, $(\nu_{x})_{x\in \MM}$ is also $\Gamma$-equivariant: 
$\gamma_{*}\nu_{x} =\nu_{\gamma x}$ for each $\gamma\in\Gamma$.
Moreover, $(\nu_{x})_{x\in\MM}$ is absolutely continuous and we denote the Radon-Nikodym derivative, called the \textit{Martin kernel}, by
\begin{equation*}
\kk(x, y, \xi):=\frac{d\nu_{y}}{d\nu_{x}}(\xi).
\end{equation*}

The Martin kernel is also characterized by the limiting behavior of the Green function.
\begin{prop} (\cite{AndSch_1985})
For each sequence $(z_{n})$ in $\MM$ with $z_{n} \to \xi \in \partial\MM$,
\begin{equation*}
\kk(x, y, \xi) = \lim_{n\to\infty} \frac{\G(y, z_{n})}{\G(x, z_{n})}.
\end{equation*}
\end{prop}

We introduce another invariant of the Brownian motion called the \textit{stochastic entropy} of the Brownian motion denoted by $\hs$. 
The stochastic entropy was first introduced by V. Kaimanovich in \cite{Kai_1986}  for co-compact manifolds with negative curvature. 
The stochastic entropy determines whether the Poisson boundary is trivial or not.
The argument in \cite{Led} easily extends to manifolds with finite volume.
\begin{prop} 
For each $x\in\MM$, $\mathbb{P}_{x}$-a.e. $\widetilde\omega$, the following limits exist and coincide:
\begin{align*}
\hs
&= \lim_{t\to\infty} -\frac{1}{t}\log \wp(t,x,\widetilde\omega_{t})
\\
&= \lim_{t\to\infty} -\frac{1}{t} \log \G(x,\widetilde\omega_{t})
.
\end{align*}
\end{prop}

Note that $\hs = (d-1)^{2}a^{2}$ when $\textrm{sec}_{\M}=-a^{2}$.
There is another characterization of the stochastic entropy analogous to the definition of the topological entropy as the exponential growth of dynamically separated sets (see \cite{Kai_1986}, \cite{Led}).

\begin{prop}
For $x\in\MM$, $T>0$ and $0<\delta<1$, 
\begin{equation*}
\hs=\lim_{T\to\infty}\frac{1}{T}\log N(x, T, \delta),
\end{equation*} 
where 
$N(x, T, \delta):=\inf \left\{ \card(E): \mathbb{P}_{x}[ d(\widetilde\omega _T, E)\le 1] \ge \delta\right\}$.
\end{prop}

\begin{proof}
Fix $\varepsilon>0$. Let 
\begin{align*}
\mathscr{C}_{T,x}
&:=
\{
\widetilde\omega_{0}=x, 
\wp(T, \widetilde\omega_{0}, \widetilde\omega_{T}) 
\le 
e^{-T(\hs-\varepsilon)}
\},\\
\mathscr{D}_{T, x}
&:=
\{
\widetilde\omega: 
\dist(\widetilde\omega_{t}, \gamma_{\theta(\widetilde\omega, \infty)}(\ld T))\le \varepsilon T, 
\wp(T, x, \gamma_{\theta(\widetilde\omega, \infty)}(\ld T))
\ge 
e^{-T(h+\varepsilon)}
\}.
\end{align*}
Choose a sufficiently large $T$ such that
$1-\frac{\delta}{2}
\le 
\mathbb{P}_{x}(\mathscr{C}_{T,x})=\mathbb{P}_{x}[\widetilde\omega_{T}\in\pi_{T}\mathscr{C}_{T,x}]$. 
We denote by $\mathbb{E}_{x}$ the expectation with respect to $\mathbb{P}_{x}$.
For each finite set $E$ such that 
$\mathbb{P}_{ x}[ d(\widetilde\omega _T, E)\le 1]\ge\delta$,
\begin{align*}
\delta 
\le\mathbb{E}_{x}[d(\widetilde\omega_{T}, E)\le1] 
&=\mathbb{P}_{x}[\{d(\widetilde\omega_{T}, E)\le1\}\cap\mathscr{C}_{T,x}]
+\mathbb{P}_{x}[\{d(\widetilde\omega_{T}, E)\le1\}\setminus\mathscr{C}_{T,x}]\\
&\le e^{-T(\hs-\varepsilon)}\sum_{y\in E} \vol B(y,1) + 1-(1-\frac{\delta}{2})\\
&\le C e^{-T(\hs-\varepsilon)} \card(E) + \frac{\delta}{2},
\end{align*}
where $C= \sup_{z} \vol B(z,1)$.
Thus, $\frac{\delta}{2C} e^{T(\hs-\varepsilon)}\le \card(E)$ and we have
\begin{equation*}
\hs \le \lim_{T\to\infty} \frac{1}{T}\log N(x,T,\delta).
\end{equation*}

For the converse inequality, Let $E$ be a minimal set satisfying $d(\widetilde\omega_{T}, E)\le 1$ for every $\widetilde\omega\in\mathscr{D}_{T, x}$ and $F\subset \{\gamma_{\theta(\widetilde\omega, \infty)}(\ld T): \widetilde\omega\in\mathscr{D}_{T, x}\}$ a maximal $\frac{1}{2}$-separated set.
Note that $\card (E) \ge N(x, T, \mathbb{P}_{x}(\mathscr{D}_{T, x}))$ and $\card (F) \le C' e^{T(h+\varepsilon)}$.
For each $f\in F$, 
\begin{equation*}
N(f)
:= \{e \in E: 
\exists\,\widetilde\omega \in \mathscr{D}_{T, x} 
\textrm{ s.t. } 
\dist(f, \gamma_{\theta(\widetilde\omega, \infty)}(\ld T)) \le \frac{1}{2}, 
\, 
\dist(\widetilde\omega_{T}, e)\le1 \}.
\end{equation*}
Then $\card N(f) \le e^{C'' \varepsilon T}$. Therefore, we have
\begin{equation*}
N(x, T, \mathbb{P}_{x}(\mathscr{D}_{T, x}))
\le
\card(E)
\le
e^{C''\varepsilon T}\card (F)
\le 
C'
e^{T [h+ (2+ C'')\varepsilon]}.
\end{equation*}
Given $\delta$, for each $T$ large enough, $N(x, T, \delta)\le N(x, T, \mathscr{D}_{T.,x})$.
\end{proof}

The stochastic entropy is related to the spectral information of $\MM$, 
the bottom of the spectrum $\lambda_{0}:=\inf \mathrm{Spec}(\Delta_{\MM})$ of the Laplacian on $\MM$. 
Note that $\lambda_{0}=(d-1)^{2}a^{2}/4$ if $\MM$ has constant negative curvature $-a^{2}$.
It was proved in Proposition 3 of \cite{Led_1990} for co-compact manifolds. 
The proof is valid for pinched negative curvature and even the co-finiteness is not required.

\begin{prop} 
$4\lambda_{0}\le\hs.$
\end{prop}

\begin{proof}
Since $\wp(t, x, y)$ is a solution of the heat equation,
\begin{align*}
\wp(t, x, y)\log \wp(t, x, y)
&=
\int_{0}^{t}
\frac{\partial}{\partial s}
\left(
\wp(s, x, y)\log \wp(s, x, y)
\right)
ds\\
&=
\int_{0}^{t}
(1+\log \wp(s, x, y))\frac{\partial}{\partial s}\wp(s, x, y)
ds\\
&=\int_{0}^{t}
(1+\log \wp(s, x, y))\Delta_{y}\wp(s, x, y)
ds.
\end{align*}
By applying this equation,
\begin{align*}
\hs
&=
\lim_{t\to \infty}
-\frac{1}{t}
\int_{\MM} \wp(t, x, y)\log \wp(t, x, y) d\vol(y)\\
&=
\lim_{t\to\infty}
\frac{1}{t}
\int_{0}^{t}
\int_{\MM}
\langle
\nabla \log \wp(s, x, y),
\nabla \wp(s, x, y)
\rangle_{g}
d\vol(y)
ds\\
&=
\lim_{t\to\infty}
\frac{4}{t}
\int_{0}^{t}
\int_{\MM}
\left\| 
\nabla \sqrt{\wp(s, x, y)} 
\right\|^{2}
d\vol(y)
ds\\
&\ge 
\frac{4}{t}\int_{0}^{t}\lambda_{0} ds
=4\lambda_{0}.
\end{align*}
The inequality is due to Rayleigh's theorem.
\end{proof}

\subsection{Thermodynamic formalisms in pinched negative curvature}
We provide some general theory of thermodynamic formalisms for geodesic flows in pinched negative curvature. 
Notions and detailed arguments can be found in \cite{PPS}.
A function $F$ on $\T^{1}\M$ is called a \textit{potential} on $\T^{1}\M$ if it is bounded and H\"older continuous. 
For a $\g^{t}$-invariant Borel probability measure $\mu$, 
if $h_{\mu}$ is the measure-theoretic entropy of the dynamical system $(\T^{1}\M, \g^{1}, \mu)$, 
we denote the pressure of $F$ for $\mu$ by $P(F, \mu)$:
\begin{equation*}
P(F, \mu)= h_{\mu} +\int_{\T^{1}\M} F d\mu.
\end{equation*}
An \textit{equilibrium state} $\mu_{F}$ for $F$ is a $\g^{t}$-invariant Borel probability measure of maximal pressure:
\begin{equation*}
P(F,\mu_{F})=\sup \, P(F, \mu)
\end{equation*}
where the supremum is taken among $\g^{t}$-invariant Borel probability measures $\mu$ s.t. $F_{-}:= \max\{ -F, 0\}$.
We denote the supremum by $P_{F}$.

Given a potential $F$ on $\T^{1}\M$, we denote the lift to $\T^{1}\MM$ by $\widetilde{F}$.
We define a line integral of a potential by
\begin{equation*}
\int_{x}^{y} \widetilde{F}
:= \int_{0}^{\dist(x,y)} \widetilde{F}(\g^{t}\rv_{x}^{y})dt,
\end{equation*}
where $\rv_{x}^{y}\in \T^{1}_{x}\MM$ is the unit vector at $x$ pointing $y$: $\gamma_{\rv_{x}^{y}}(\dist(x,y))= y$.
A \textit{Patterson-Sullivan density} for $F$ of dimension $\delta$ is a family $(\mu_{x})_{x\in\MM}$ of finite Borel measures absolutely continuous to each other on $\partial\MM$ satisfying
\begin{align*}
\gamma_{*}\mu_{x}
&=\mu_{\gamma x},
\\
d\mu_{y}(\xi)
&=\exp\left( C_{F-\delta}(x, y, \xi)\right)d\mu_{x}(\xi),
\end{align*}
for each 
$x, y \in \MM$, $\gamma \in \Gamma$
where 
\begin{equation*}
C_{F}(x, y, \xi)
:=\lim_{z \to \xi} \int_{y}^{z}\widetilde{F} - \int_{x}^{z} \widetilde{F}. 
\end{equation*}
We denote by $\mu^{\T}_{x}$ the \textit{spherical measure} at $x$, the push-forward measure of $\mu_{x}$ via the inverse of homeomorphism $\T^{1}_{x}\MM \to \partial\MM$ for each $x\in\MM$.

Let $\rv\in\T^{1}\M$ with a lift $\widetilde\rv$ to a vector in $\T^{1}\MM$. 
Define the \textit{Bowen ball} around $\rv$ by
\begin{equation*}
B(\rv, T, T', r)
:= 
\{ \mathrm{w}\in \T^{1}\M: 
\sup_{t\in[-T', T]}\dist( \gamma_{\widetilde\rv}(t), \gamma_{\widetilde{\mathrm{w}}}(t))<r,
\, \exists \,\textrm{a lift }\widetilde\rw\in\T^{1}\MM
\}.
\end{equation*}
One can construct a Gibbs measure from a Patterson-Sullivan density.
That is, if a Patterson-Sullivan density $(\mu_{x})$ for $F$ of dimension $P_{F}$ is given, 
there is a $\g^{t}$-invariant Borel measure $\widetilde\mu$ on $\T^{1}\MM$ which is $\Gamma$-invariant and whose induced measure $\mu$ on $\T^{1}\M$ has a Gibbs property (see Section 3.8 of \cite{PPS}):
For each compact set $K\in\T^{1}\MM$, there exist $r>0$ and $c_{K, r}>0$ such that for every $T, T'\ge0$
and for every $\rv$,
\begin{equation*}
c_{K, r}^{-1}\exp\int_{-T'}^{T} \left(F(\g^{t}\rv)-P_{F}\right)dt
\le
\mu(B(\rv, T, T', r)
\le 
c_{K, r} \exp\int_{-T'}^{T} \left(F(\g^{t}\rv)-P_{F}\right) dt.
\end{equation*}
We call $\widetilde\mu$ the \textit{Gibbs measure} of $F$ and $(\mu_{x})$. 
The Gibbs measure determines whether an equilibrium state for $F$ exists or not.

\begin{prop}\label{VP}(\cite{PPS})
$F$ is H\"older continous with $P_{F}<\infty$. 
\begin{enumerate}
\item
there is a Patterson-Sullivan density $(\mu_{x})$ for $F$ of dimension $P_{F}$ unique up to multiplicative constants.
\item
If the corresponding Gibbs measure $\widetilde\mu_{F}$ induces a finite measure $\mu_{F}$ on $\T^{1}\M$ then $\mu_{F}$ is the unique equilibrium state for $F$ and $\mu_{F}$ is ergodic.
Otherwise, there is no equilibrium state for $F$.
\end{enumerate}
\end{prop}

V. Pit and B. Schapira found a necessary and sufficient condition for the finiteness of Gibbs measure in \cite{PS_2018}. One can find the same statement also in \cite{PPS}.

\begin{prop} \label{SPR}
A H\"older continuous potential $F$ admits an equilibrium state if and only if for every maximal parabolic subgroup $\Pi$ of $\Gamma$, the following series converges:
\begin{equation*}
\sum_{\gamma\in\Pi} 
\dist_{}(x, \gamma x)
\exp
\int_{x}^{\gamma x} (\widetilde{F}-P_{F}).
\end{equation*}
\end{prop}

We have an ergodic theorem for the geodesic flow with respect to spherical measures. 
We also derive a Gibbs property for spherical measures (see \cite{Led_1988}).

\begin{prop} \label{transversalerg}
If a bounded H\"older continuous potential $F$ admits an equilibrium state $\mu$ then
for every $\phi\in\C_{b}(\T^{1}\M)$, $x\in \M$ and for $\mu^{\T}_{x}$-a.e. $\rv$ in $\T^{1}\M$,
\begin{align}
\label{transversalpterg}
\frac{1}{t}
\int_{0}^{t} \phi(\g^{s}\rv) ds 
\to 
\int_{\T^{1}\M} \phi \, d \mu
&\, 
\textrm{ as }t \to \infty,
\\
\label{transversalmeasure}
\lim_{t\to\infty} 
-\frac{1}{t} 
\log \mu^{\T}_{x}(B(\rv, t, 0, \varepsilon))
=
h_{\mu}
&\,
\textrm{ for some } \varepsilon>0.
\end{align}
\end{prop}

\begin{proof}
Since $\mu$ is ergodic, 
the set $G$ of the vectors for which the convergence (\ref{transversalpterg}) holds is a union of stable leaves with $\mu(G)=1$.
Thus for any $x, y \in \M$, the projections $G_{x}^{+}:=\{\widetilde{\rv}^{+}: \rv\in\T^{1}_{x}\M\}$ and $G_{y}^{+}$ of fiber onto the boundary at infinity $\partial_{\infty}\MM$ are identical. 
Since $\mu_{x}^{\T}(G\cap \T^{1}_{x}\M) = \mu_{x}(G_{x}^{+})$, $G\cap \T^{1}_{x}\M$ is a $\mu_{x}^{\T}$-full set if and only if $G\cap\T^{1}_{y}\M$ is a $\mu_{y}^{\T}$-full set. 
Therefore $G\cap \T^{1}_{x}\M$ is a $\mu^{\T}_{x}$-full set for every $x\in\M$.

From the $P_{F}$-Gibbs property of $\mu$, for $\mu^{\T}_{x}$-a.e. $\rv\in G\cap \T^{1}_{x}\M$, 
\begin{align}\label{Gibbs}
\lim_{t\to\infty} 
-\frac{1}{t} \log \mu (B(\rv, t, 0, \varepsilon))
=
P_{F} 
-
\lim_{t\to\infty}
\frac{1}{t}\int_{0}^{t} F(\g^{s}\rv)ds
=
P_{F} - \int Fd\mu
= h_{\mu}.
\end{align}
A local stable manifold of $\rv \in \T^{1}\M$ is 
\begin{equation*}
W^{s}_{\varepsilon}(\rv):=\{ \rw: \dist(\g^{t}\rv, \g^{t}\rw)<\varepsilon, \, \forall t\ge0\}.
\end{equation*}
The spherical measure is a transversal measure, so it can be defined by local stable manifolds: 
\begin{equation*}
\mu_{x}^{\T} (S) = \mu\left( \cup_{\rw\in S} W^{s}_{\varepsilon}(\rw)\right).
\end{equation*}
Since the Bowen ball consists of local stable manifolds, (\ref{Gibbs}) holds when we replace $\mu$ by $\mu^{\T}_{x}$.
\end{proof}

There are two important potentials. 
The first is the zero potential, whose equilibrium state is the measure of maximal entropy, 
also called the \textit{Bowen-Margulis measure} if it admits an equilibrium state. 
The measure class of the Patterson-Sullivan density for the zero potential is called the \textit{visibility class}.

The other is the \textit{geometric potential} $\J$ induces from the $\Gamma$-invariant function
\begin{equation*}
\widetilde\J(\rv)= -\left.\frac{d}{dt}\right|_{t=0} \log \det \T_{\rv}\g^{t}|_{E^{su}(\rv)}
\end{equation*}
on $\MM$, 
where $\T_{\rv}\g^{t}:\T_{\rv}\T^{1}\MM \to \T_{\g^{t}\rv}\T^{1}\MM$ is the tangent map of the flow map $\g^{t}$ at $\rv$ and $E^{su}(\rv)= \T_{\rv}\W^{su}(\rv)$ is the strong unstable distribution. 
Due to the pinched negative curvature and the uniform bound on the first derivatives of the sectional curvature,
the angles between the stable leaves and the strong unstable leaves have positive lower bound 
and 
the foliations are H\"older continuous.
Thus $\J$ is H\"older continuous and the Liouville measure on $\T^{1}\M$ is the equilibrium state for $\J$.
The existence with an assumption on the pressure of $\J$ is proved in Chapter 7 of \cite{PPS} and \cite{Riq_2018} proves that the assumption is true in our case.
The measure class determined by the Patterson-Sullivan density for the geometric potential is called the \textit{Lebesgue class}.

\section{Central limit theorem of Brownian motions}

\subsection{Foliated Brownian motions}
We shall introduce a Markov process on $\T^{1}\M$ called the foliated Brownian motion for the stable foliation of $\T^{1}\M$. 
The foliated Brownian motion was first introduced in the way to develop the ergodic theory of foliations (See \cite{CanCon}, \cite{Garnett_1983}).

Fix a fundamental domain 
$\M_{0}\subset\MM$ of $\Gamma$. 
Identify $\M$, $\T^{1}\M$ with $\M_{0}$, $\M_{0}\times\partial\MM$, respectively. 
Note that $\widetilde\W^{s}(x, \xi)=\MM\times\{\xi\}$ is projected onto 
\begin{equation*}
\W^{s}(x, \xi):=\{(y, \gamma^{-1} \xi) \in \T^{1}\M : y\in\M_{0}, \, \gamma \in \Gamma \}.
\end{equation*}
The stable foliation $\W^{s}=\{\W^{s}(\rv):\rv\in\T^{1}\M\}$ of $\T^{1}\M$ is the collection of the projected stable leaves.
Similarly, we define the stable distribution $E^{s}$ of $\T^{1}\M$.
The stable leaves of $\W^{s}$ inherit the Riemannian metric from $g_{s}$ on the leaves of $\widetilde\W^{s}$ which is also denoted by $g_{s}$. 
We denote the inherited differentials by $\mathrm{div}_{s}$ and $\Delta_{s}$.

\begin{defn}
Let $\PP(\T^{1}\M)$ be the space of probability measures on $\T^{1}\M$.
We define a transition semigroup $\mathbf{P}: (0, \infty)\times \T^{1}\M \to \PP(\T^{1}\M)$ 
by 
\begin{equation*}
d\,\mathbf{P}[t, \rv](\rw)
= \sum_{\gamma\in\Gamma}\wp(t, x, \gamma y) \, d\delta_{\gamma^{-1}\xi}(\eta)\left. d\vol\right|_{\M_{0}}(y),
\end{equation*}
for $\rv=(x,\xi), \rw=(y,\eta)\in \T^{1}\M$. 
The transition semigroup defines a unique family $\{\mathbb{P}_{(x,\xi)}\}_{(x,\xi)\in\T^{1}\M}$ of Borel probability measures on the space $\T^{1}\Omega:=\C(\mathbb{R}_{+}, \T^{1}\M)$ of sample paths on $\T^{1}\M$.
The canonical filtration is the collection of the smallest $\sigma$-algebras $\Fi_{t}=\Fi_{t}(\T^{1}\M):=\sigma\{\pi_{s}: 0 \le s \le t\}$ for which the projections $\pi_{s}(\omega)= \omega_{s}$ on $\T^{1}\Omega$ are measurable. 
The canonical process $Z_{t}(\omega)=\omega_{t}$ of the filtered space $( \T^{1}\Omega, \left\{ \Fi_{t} \right\}_{0\le t\le\infty})$ is a Markov process with respect to $\mathbb{P}_{(x, \xi)}$,
which is called the foliated Brownian motion for $\W^{s}$ with initial distribution $\delta_{(x, \xi)}$,
for each $(x, \xi)\in \T^{1}\M$.
\end{defn}

We define the Markov operator
$\mathcal{Q}^{t}:\mathcal{C}_{b}(\T^{1}\M)\to\mathcal{C}_{b}(\T^{1}\M)$ 
on the space of bounded continuous functions on $\T^{1}\M$ by 
\begin{equation} \label{Markovint}
\Q^{t}f (\rv) 
:= \int_{\T^{1}\M} f d\,\mathbf{P}[t, \rv]
= \sum_{\gamma\in\Gamma}\int_{\M_{0}} f(y, \gamma^{-1} \xi) \wp(t, x, \gamma y) d\vol(y).
\end{equation}
Note that the foliated Brownian motion for $\W^{s}$ is the projected process of a Markov process, called the foliated Brownian motion for $\widetilde\W^{s}$, with the transition semigroup 
\begin{equation} \label{diffusionint}
d\, \widetilde{\mathbf{P}}[t, \rv](\rw)
=\wp(t, x,  y) d\delta_{\xi}(\eta)d\vol(y).
\end{equation}
Let $\widetilde\Q^{t}$ be the Markov operator on $\T^{1}\MM$.
For any $f\in\C_{b}(\T^{1}\M)$ and for each $(x, \xi) \in \M_{0}\times\partial\MM$, 
\begin{equation*}
\mathcal{Q}^{t}f(x, \xi)
=\int_{\MM} \widetilde{f}(y, \xi)\wp(t, x, y)d\vol(y)
=\widetilde\Q^{t}\widetilde{f}(x, \xi),
\end{equation*}
where $\widetilde{f}$ is the $\Gamma$-invariant lift of $f$ to $\T^{1}\MM$. 
Note that the infinitesimal generator of the Markov operator is the foliated Laplacian:
\begin{equation*}
\left. \frac{d}{dt}\right|_{t=0} \Q^{t} f = \Delta_{s} f.
\end{equation*}

L. Garnett proved in \cite{Garnett_1983} that the Markov operator $\widetilde\Q$ admits an invariant measure $\Hh$ on $\T^{1}\MM$ of the form
\begin{equation*}
d\Hh(x,\xi)= d\nu_{x}(\xi)d\widetilde\m(x)=\kk(x_{0}, x,\xi)d\widetilde\m(x) d\nu_{x_0}(\xi),
\end{equation*} 
where 
$\widetilde{m}= \frac{1}{\vol(\M_{0})}\vol$ and $\nu_{x}$ is the harmonic measure.
We have an induced probability measure $\hh := \Hh|_{\M_{0}\times\partial\MM}$ on $\T^{1}\M$. 
By $\Gamma$-equivariance of $\nu_{x}$, 
\begin{align*}
\int_{\T^{1}\M} \Q^{t}f d\hh
&= 
\frac{1}{\vol(\M_{0})}
\int_{\M_{0}}\int_{\partial\MM} 
\sum_{\gamma} 
\int_{\M_{0}} \widetilde{f}(y, \gamma^{-1}\xi)\wp(t, x, \gamma y)
d\vol(y) d\nu_{x}(\xi) d\vol(x)
\\
&=
\frac{1}{\vol(\M_{0})}
\int_{\M_{0}}
\sum_{\gamma}
\int_{\M_{0}}\int_{\partial\MM} 
\widetilde{f}(y, \xi)\wp(t, \gamma^{-1}x, y)
d\nu_{\gamma^{-1}x}(\xi) d\vol(x) 
d\vol(y).
\end{align*}
Since we know $d\nu_{\gamma^{-1}x}(\xi) = \kk(y, \gamma^{-1} x, \xi)d\nu_{y}(\xi)$, 
the integrand in the right-handed side is:
\begin{align*}
&\sum_{\gamma}
\int_{\M_{0}}\int_{\partial\MM} 
\widetilde{f}(y, \xi)\wp(t, \gamma^{-1}x, y)
d\nu_{\gamma^{-1}x}(\xi) d\vol(x) \\
&= 
\int_{\partial\MM}
\widetilde{f}(y, \xi) 
\sum_{\gamma}
\int_{M_{0}}\wp(t, \gamma^{-1} x, y) \kk(y, \gamma^{-1} x, \xi) 
d\vol(x) d\nu_{y}(\xi)\\
&=
\int_{\partial\MM}
\widetilde{f}(y, \xi) 
\int_{\MM}
\wp(t, x, y)\kk(y, x, \xi) 
d\vol(x) d\nu_{y}(\xi)\\
&=\int_{\partial\MM} \widetilde{f}(x, \xi) d\nu_{y}(\xi).
\end{align*}
We used the harmonicity of the Martin kernel in the last equality:
\begin{equation*}
\int_{\MM}\wp(t, x, y)\kk(y, x, \xi) d\vol(x)= \kk(y, y, \xi)=1.
\end{equation*}
Therefore, we have the $\Q^{t}$-invariance of $\hh$. 
The stationary measure of the foliated Brownian motion is $\mathbb{P}_{\hh}=\int_{\T^{1}\M} \mathbb{P}_{(x, \xi)}d\hh(x, \xi)$ and is ergodic for the shift map on $\T^{1}\Omega$.

We have an integral expression of the linear drift and the stochastic entropy. 
Propsosition 2.9 and 2.16 in \cite{LedShu_2017} prove the same descriptions for the Brownian motion on co-compact negatively curved manifolds.
The identities for co-finite manifolds follow in the same way.

\begin{prop} \label{intcha}
$\ld^{2}\le\hs$. Moreover,
\begin{align*}
\ld 
&= 
\int_{\M_{0}}\int_{\partial\MM} \Delta_{y} \bb(y, x, \xi) d\nu_{y}(\xi) d\widetilde\m(y)
\\
&= 
\int_{\M_{0}} 
\int_{\partial\MM}
\langle  -\nabla_{y} \bb(y, x, \xi),\nabla_{y} \log \kk(x,y, \xi) \rangle_{g}
\,
d\nu_{y}(\xi) 
\,
d\widetilde\m(y),
\end{align*}
\textrm{ and }
\begin{align*}
\hs
&=
\int_{\M_{0}} 
\int_{\partial\MM}
|\nabla_{y} \log \kk(x,y, \xi)|^2 
\,
d\nu_{y}(\xi) 
\,
d\widetilde\m(y)
.
\end{align*}
\end{prop}

\begin{proof}
We only verify the second equality. 
The other equalities follow immediately from the same argument as in \cite{LedShu_2017}.
\begin{align*}
\ld 
&= \int_{\M_{0}}\int_{\partial\MM} \Delta_{y} \bb(y, x, \xi) d\nu_{y}(\xi) d\widetilde\m(y)
\\
&= \int_{\partial\MM}\int_{\M_{0}} 
\Delta_{y}\bb(y, x, \xi) \kk(x, y, \xi) 
d\widetilde\m(y)
d\nu_{x}(\xi)
\\
&= \int_{\partial\MM}\int_{\M_{0}}
\langle -\nabla_{y} \bb(y, x, \xi), \nabla_{y} \kk(x, y, \xi) \rangle_{g}
d\widetilde\m(y)
d\nu_{x}(\xi)
\\
&=
\int_{\M_{0}}
\int_{\partial\MM}
\langle -\nabla_{y} \bb(y, x, \xi), \nabla_{y}\log\kk(x, y, \xi)\rangle
d\nu_{y}(\xi)
d\widetilde\m(y).
\end{align*}
\end{proof}

\subsection{Leafwise heat equation}
We prove the contraction property on  H\"older spaces of the foliated Brownian motion.
Let $\tau>0$. We define a $\tau$-H\"older norm of $f$ in the space $\C_{b}(\T^{1}\M)$ of bounded continuous functions by
\begin{equation*}
\|f\|_{\mathcal{L}^{\tau}}
=\|f\|_{\infty}
+\sup_{x\in\M_{0}} 
\sup_{\xi,\eta\in\partial\MM}
\frac{|\widetilde{f}(x,\xi)-\widetilde{f}(x,\eta)|}{d_{\infty}^{x,\tau}(\xi, \eta)},
\end{equation*} 
and we denote the corresponding H\"older space by
\begin{equation*}
\mathcal{L}^{\tau}=\{f\in\mathcal{C}_{b}( \T^{1}\M): \|f\|_{\mathcal{L}^{\tau}}<\infty\}.
\end{equation*}
The following statement corresponds to the uniqueness of a $\Q^{t}$-invariant measure for compact negatively curved manifolds (see \cite{Led_1995}).
In \cite{Ham_1997}, it was shown that the uniqueness for the $(\Delta_{s}+Y)$-diffusion on compact negatively curved manifolds holds for a stably closed vector field $Y$ on $\T^{1}\M$ with positive pressure.

\begin{prop} 
For every $\Q^{t}$-invariant measure $\eta$ on $\T^{1}\M$ and for each $f\in\mathcal{L}^{\tau}$,
$$\int f d\eta = \int f d\hh.$$
\end{prop}
\begin{proof}
If $\eta$ is a $\Q^{t}$-invariant measure on $\T^{1}\M$, its $\Gamma$-invariant lift $\widetilde\eta$ to $\T^{1}\MM$ is disintegrated into $d\widetilde\eta(x,\xi)=d\widetilde\eta_{x}(\xi)d\widetilde\m(x)$ over the fibration $\T^{1}\MM=\MM\times\partial\MM$ where $\widetilde\eta_{x}$ are the conditional measures on the unit spheres $\T_{x}\MM=\{x\}\times\partial\MM$ of $\widetilde\eta$ (\cite{Garnett_1983}). 
As in the proof of Proposition \ref{transversalerg}, we can consider $\widetilde\eta_{x}$ as a probability measure of the union of local leaves; for some sufficiently small $\delta>0$,
\begin{equation}
\widetilde\eta_{x}(A):= \widetilde\eta( \cup_{\rw\in A} W_{\delta}^{s}(\rw)) / \widetilde\eta(\cup_{\rv \in \T^{1}_{x}\MM}W_{\delta}^{s}(\rv)).
\end{equation}
We denote by $\mathbb{E}_{x}$ the expectation with respect to $\mathbb{P}_{x}$. 
From the $\Q^{t}$-invariance, we have
\begin{align*}
\int_{\T^{1}\M}f d\eta
&=
\int_{\M_{0}}
\int_{\partial\MM} 
\Q^{t} f(x,\xi)
\,
d\widetilde\eta_{x}(\xi)
\,
d\widetilde\m(x)\\
&=
\int_{\M_{0}}
\int_{\partial\MM}
\int_{\MM}
\wp(t,x,y) 
\widetilde{f}(y, \xi)
\,
d\vol(y)
\,
d\widetilde\eta_{x}(\xi)
\,
d\widetilde\m(x)\\
&=
\int_{\M_{0}}
\mathbb{E}_{x}
\left[
\int_{\partial\MM} \widetilde{f}(\widetilde\omega_{t}, \xi) 
\,
d\widetilde\eta_{x}(\xi)
\right]
\,
d\widetilde\m(x).
\end{align*}
Note that given $\varepsilon>0$, $x\in\M_{0}$ and $f\in\mathcal{L}^{\tau}$, 
there is $\theta_{0}=\theta_{0}(\varepsilon)>0$ for every $y\in\MM$ and $\xi, \eta \in \partial\MM$ with $\angle_{y}(\xi, \eta)<\theta_{0}$,
$
|\widetilde{f}(y, \xi)-\widetilde{f}(y, \eta)|
< 
\varepsilon.$
Given $\xi \in \partial\MM$, we set for $T, \theta >0$,
\begin{align*}
\Upsilon(x, \xi, \theta)
&:= \left\{ \widetilde\omega: \angle_{x}( \widetilde\omega_{\infty}, \xi ) <\frac{\theta}{3} \right\},\\
\Xi(x, T, \theta) 
&:= \left\{ \widetilde\omega: \dist(x, \widetilde\omega_{t}) \ge \frac{\ld}{2}t, \angle_{x}( \theta (\widetilde\omega, t), \theta(\widetilde\omega, \infty) ) < \frac{\theta}{3}, \, \forall t \ge T \right\}.
\end{align*}
Then if $\theta\in(0,\theta_{0})$ is small enough, than for any $x$ and $\xi$, $\mathbb{P}_{x} (\Upsilon(x, \xi, \theta) ) <\frac{\varepsilon}{2\|f\|_{\infty}}$ (by \cite{BenHul_2019}). 
Choose such a small $\theta$. 
There is $T_{0}=T_{0}(x, \theta)$ such that if $t>T_{0}$, 
$|\widetilde{f}(\widetilde\omega_{t}, \xi) -\widetilde{f} (\rv_{\widetilde\omega_{t}}^{x})|<\frac{\varepsilon}{2\|f\|_{\infty}}$ for each $\widetilde\omega\in\Xi(x, t, \theta) \setminus \Upsilon(x, \xi ,\theta)$ and
$\mathbb{P}_{x} (\Xi(x, t, \theta) ) >1-\frac{\varepsilon}{2\|f\|_{\infty}}$. 
Hence if $t>T_{0}$, then
\begin{align*}
&\left|
\mathbb{E}_{x}
\left[
\int_{\partial\MM} 
\widetilde{f}(\widetilde\omega_{t}, \xi) 
-
\widetilde{f} (\rv_{\widetilde\omega_{t}}^{x})
\,
d\widetilde\eta_{x}(\xi)
\right]
\right|\\
&\le
\int_{\partial\MM} 
\mathbb{E}_{x} 
\left[ 
\left| 
\widetilde{f}(\widetilde\omega_{t}, \xi) -\widetilde{f} (\rv_{\widetilde\omega_{t}}^{x})
\right| 
\left(
\mathbf{1}_{\Xi(x, t, \delta) \setminus \Upsilon(x, \xi, \theta)}
+
\mathbf{1}_{\Xi(x, t, \delta)\cap \Upsilon(x, \xi, \theta)}
+\mathbf{1}_{\Xi(x, t, \delta)^{c}}
\right)
\right]
d \widetilde\eta_{x}(\xi)
\\
&\le
\varepsilon
+2\|f\|_{\infty}
\left(
\int_{\partial\MM}
\mathbb{P}_{x}
\left[
\Upsilon(x, \xi, \theta)
\right]
d \widetilde\eta_{x}(\xi)
+
\mathbb{P}_{x}
\left[
\Xi(x, t, \delta)^{c}
\right]
\right)\\
&<
 3\varepsilon.
\end{align*}
Since $\phi_{t}(x):= \mathbb{E}_{x}\left[ \int \widetilde{f}(\widetilde\omega_{t}, \cdot)d\eta_{x}\right]$ and $\psi_{t}(x):=\mathbb{E}_{x}\left[\widetilde{f}(\rv_{\widetilde\omega_{t}}^{x})\right]$ are bounded by $\|f\|_{\infty}$, it follows that $\phi_{t} - \psi_{t} \to 0$ as $t\to \infty$ in $L^{1}(\M_{0}, \widetilde\m)$ and hence $ \underset{t\to\infty}{\lim}\int \psi_{t} d\widetilde\m = \int f d\eta$.
Thus we have
\begin{align*}
\int f d\eta
=\lim_{t\to\infty}
\int_{\M_{0}}
\mathbb{E}_{x}
\left[
\int_{\partial\MM} \widetilde{f}(\widetilde\omega_{t}, \xi) 
\,
d\widetilde\eta_{x}(\xi)
\right]
\,
d\widetilde\m(x)
=
\lim_{t\to\infty}
\int_{\M_{0}}
\mathbb{E}_{x}
\left[
\widetilde{f}(\rv_{\widetilde\omega_{t}}^{x})
\right]
\,
d\widetilde\m(x).
\end{align*}
From the $\Gamma$-invariance of $\widetilde{f}$ and the heat kernel, it follows that
\begin{align*}
\int_{\M_{0}}
\mathbb{E}_{x}
\left[
\widetilde{f}(\rv_{\widetilde\omega_{t}}^{x})
\right]
\,
d\widetilde\m(x)
&=
\int_{\M_{0}}
\int_{\M_{0}}
\sum_{\gamma\in\Gamma}
\wp(t,x,\gamma y)
\widetilde{f}(\rv_{\gamma y}^{x})
\,
d\vol(y)
\,
d\widetilde\m(x)\\
&=
\int_{\M_{0}}
\int_{\M_{0}}
\sum_{\gamma\in\Gamma}
\wp(t,y, \gamma^{-1} x)
\widetilde{f}(\rv_{y}^{\gamma^{-1}x})
\,
d\vol(x)
\,
d\widetilde\m(y)\\
&=
\int_{\M_{0}}
\mathbb{E}_{y}
\left[
\widetilde{f}(\rv_{y}^{\widetilde\omega_{t}})
\right]
\,
d\widetilde\m(y)
.
\end{align*}
Letting $t$ tend to infinity, we have
\begin{align*}
\int f d\eta
&=
\int_{\M_{0}}
\mathbb{E}_{y}
\left[
\widetilde{f}(y, \widetilde\omega_{\infty})
\right]
\,
d\widetilde\m(y)\\
&=
\int_{M_{0}}
\int_{\partial\MM}
\widetilde{f}(y, \xi)
\,
d\nu_{y}(\xi)
\,
d\widetilde\m(y).
\end{align*}
Therefore, $\int f d\eta = \int f d\hh$.
\end{proof}

We denote by $\N$ the integration operator on $\C_{b}(\T^{1}\M)$: 
\begin{equation*}
\N(f)
:=
\int_{\T^{1}\M}
f
\,d \hh.
\end{equation*}
The Markov operator $\Q^{t}$ converges to $\N$ on $\LL^{\tau}$.
Furthermore the following theorem shows the rate of convergence is exponentially fast. 
We postpone the proof until Section \ref{ProofCtr}.

\begin{thm}\label{Ctr}
$\mathcal{Q}^{t}: \mathcal{L}^{\tau}\to\mathcal{L}^{\tau}$ defines a one-parameter semigroup of continuous operators for small enough $\tau>0$. Furthermore, there is $C=C(\tau)>0$ such that for every $t>0$,
\begin{equation*}
\| \mathcal{Q}^{t}-\mathcal{N}\|_{\mathcal{L}^{\tau}}
\le 
e^{-Ct}.
\end{equation*}
\end{thm}

Given $f\in \LL^{\tau}$, if $\int f d\hh=0$, then the $\LL^{\tau}$-limit of $\int_{0}^{T}\Q^{t}f dt$ exists by the contraction property. 
The limit $u:= \lim_{T\to\infty}\int_{0}^{T} \Q^{t}f dt$ is a weak solution of the leafwise heat equation $\Delta_{s} u = -f$, thus a strong solution in $\LL^{\tau}$. 
Since a leafwise harmonic $u$ is $\Q^{t}$-invariant, the uniqueness also follows from the contraction property (See \cite{Led_1995} for the detail).
Therefore we obtain the following corollary.

\begin{coro} \label{Heq}
For small enough $\tau>0$ and every $f\in\LL^{\tau}$ with $\int fd\hh=0$, 
there exists a solution $u \in \LL^{\tau}$ to the leafwise heat equation $\Delta_{s} u=-f$ 
which is unique up to additive constants. 
In addition, $u$ is $\mathcal{C}^{2}$ along the stable leaves.
\end{coro}

Let $\alpha: \T^{1}\M\to E^{s*}$ be a continuous section of the dual bundle $E^{s*}$ of the stable distribution $E^{s}$ of $\T^{1}\M$ and $\widetilde\alpha : \T^{1}\MM \to \widetilde{E}^{s}$ be the lift of $\alpha$.
The section $\alpha$ is called a \textit{leafwise closed 1-form} of class $\C^{1}$ if 
$\widetilde\alpha|_{\widetilde\W^{s}(\rv)}$ is a closed 1-form on $\widetilde\W^{s}(\rv)$ of class $\C^{1}$ for any $\rv\in\T^{1}\MM$.
For each $(x, \xi)\in \T^{1}\MM$, since $\widetilde\W^{s}(x,\xi)=\MM\times\{\xi\}$ is diffeomorphic to $\MM$,
there is a 1-form $\widetilde\alpha^{\xi}$ on $\MM$ which agrees with the pull-back of $\widetilde\alpha|_{\widetilde\W^{s}(x,\xi)}$.
Furthermore, if $\alpha$ is a leafwise closed 1-form of class $\C^{1}$, then there exists $A^{\xi}\in\C^{1}(\MM)$ such that $dA^{\xi}=\widetilde\alpha^{\xi}$.
Hence if $\alpha$ is a leafwise closed 1-form of class $\C^{1}$, 
we define for each foliated Brownian path $\omega \in \T^{1}\Omega$ starting from $(x,\xi)\in\T^{1}\M$,
\begin{equation*}
\int_{\omega_{0}}^{\omega_{t}}\alpha:=A^{\xi}(\widetilde\omega_{t})-A^{\xi}(\widetilde\omega_{0})
\end{equation*}
for every $t\ge0$, where $\widetilde\omega$ is a Brownian path on $\T^{1}\MM$ such that
$(\widetilde\omega_{t}, \xi)\in \T^{1}\MM$ is a lift of $\omega_{t}$.

We denote by $\delta_{s}$ the leafwise codifferential $g_{s}$-dual to $-\mathrm{div}_{s}$, 
that is, $\delta_{s}\alpha = -\mathrm{div}_{s}\alpha^{\#}$ where $\alpha^{\#}:\T^{1}\M\to E^{s}$ is the continuous section $g_{s}$-dual to $\alpha$. 
Since 
\begin{equation*}
\delta_{s}\widetilde\alpha(x,\xi)
= 
-\mathrm{div}_{s}\widetilde\alpha^{\#}(x, \xi)
=
-\mathrm{div}\nabla A^{\xi}(x)
=
-\Delta A^{\xi}(x) 
=
-\Delta_{s} A(x, {\xi}), 
\end{equation*}
by It\^o's formula (see Chapter 3 of \cite{Hsu}), 
\begin{equation} \label{ito}
\mathbf{X}_{t}(\omega)
=\int_{\omega_{0}}^{\omega_{t}}\alpha+\int_{0}^{t}\delta_{s}\alpha(\omega_{r})dr
=A^{\xi}(\widetilde\omega_{t})-A^{\xi}(\widetilde\omega_{0})-\int_{0}^{t} \Delta A^{\xi}(\widetilde\omega_{r})dr
\end{equation}
is a martingale on $(\T^{1}\Omega, \{\Fi_{t}(\T^{1}\M)\}_{0\le t \le\infty}, \mathbb{P}_{\hh})$ having the quadratic variation 
\begin{equation*}
d\langle \mathbf{X}, \mathbf{X}\rangle_{t}(\omega)=(\Delta(A^{\xi})^{2}-2A^{\xi}\Delta A^{\xi})(\widetilde\omega_{t})dt=2\|\alpha^{\#}(\omega_{t})\|^{2}dt.
\end{equation*}
If $\beta$ is a leafwise closed 1-form of class $\C^{1}$ such that $\delta_{s}\beta$ is H\"older continuous on $\T^{1}\M$, applying Corollary \ref{Heq}, there is 
$u\in\mathcal{L}^{\tau}$
 such that 
 $\Delta_{s}u=\delta_{s}\beta-\int \delta_{s}\beta d\hh$.
 Hence, due to the equation (\ref{ito}) for $\alpha= \beta+du$, we have a martingale
 \begin{equation}\label{itohol}
 \mathbf{X}_{t}
 =\int_{\omega_{0}}^{\omega_{t}}(\beta+du)+\int_{0}^{t}\delta_{s}(\beta+du)(\omega_{r})dr
 =\int_{\omega_{0}}^{\omega_{t}}\beta +u(\omega_{t})-u(\omega_{0})+t\int\delta_{s}\beta d\hh
 \end{equation}
with the quadratic variation 
$\langle \mathbf{X}, \mathbf{X} \rangle_{t}(\omega)=2\int_{0}^{t}\|\alpha^{\#}+\nabla u\|^{2}(\omega_{t})ds$.

\subsection{Proof of Theorem \ref{CLT}}
For $(x,\xi)\in\T^{1}\MM$, let $B(x, \xi):=\bb(x, x_{0}, \xi)$, $K(x, \xi):=\log \kk(x_{0}, x, \xi)$. 
Note that 
\begin{equation*}
\Delta_{s}B(x, \xi) = \Delta_{x}\bb(x, x_{0}, \xi)
\end{equation*}
is H\"older continuous due to uniform bounds of the first derivatives of curvature. On the other hand, 
\begin{equation*}
\Delta_{s} K(x, \xi) = -\| \nabla_{x} \log \kk(x_{0}, x, \xi)\|^{2}
\end{equation*}
is H\"older contunous due to \cite{AndSch_1985}, \cite{Hamenst_dt_1990}.
By Corollary \ref{Heq} for $f= \Delta_{s}B$, $\Delta_{s}K$ 
there exist $u_{\bb}, u_{\kk} \in\mathcal{L}^{\tau}$ 
for which we obtain square-integrable martingales 
\begin{align*}
\mathbf{B}_{t}(\omega) 
= \bb(\widetilde\omega_{t}, \widetilde\omega_{0}, \xi) 
- t\ld 
+u_{\bb}(\omega_{t}) 
-u_{\bb}(\omega_{0}),
\, \, 
\mathbf{K}_{t}(\omega) 
=  \log \kk(\widetilde\omega_{0}, \widetilde\omega_{t}, \xi) 
+t\hs 
+u_{\kk}(\omega_{t}) 
-u_{\kk}(\omega_{0}), 
\end{align*}
for $\omega\in\T^{1}\Omega$ with a lift $(\widetilde\omega, \xi)\in \T^{1}\MM$,
by the It\^o formula (\ref{itohol}) for 
$\beta_{(x, \xi)}=dB^{\xi}_{x}$, $dK^{\xi}_{x}$, respectively.
Their quadratic variations are
\begin{align} \label{quad}
\langle \mathbf{B}, \mathbf{B} \rangle_{t}(\omega) 
=
2\int_{0}^{t}\| \nabla B +\nabla u_{\bb} \|^{2}(\omega_{s})ds,
\,\,
\langle\mathbf{K}, \mathbf{K}\rangle_{t}(\omega) 
=
2\int_{0}^{t}\| \nabla K + \nabla u_{\kk} \|^{2}(\omega_{s})ds. 
\end{align}

We denote by $\mathbb{E}_{(x, \xi)}$ the expectation with respect to $\mathbb{P}_{(x, \xi)}$.
From the equalities (\ref{quad}) of quadratic variations,
\begin{align*}
\mathbb{E}_{(x,\xi)}
\left[
\frac{1}{t}
\langle \mathbf{B}, \mathbf{B} \rangle_{t}(\omega)
\right]
=\frac{2}{t}\int_{0}^{t} \Q^{s}\| \nabla B +\nabla u_{\bb} \|^{2}(x, \xi) ds,
\\
\mathbb{E}_{(x,\xi)}
\left[
\frac{1}{t}
\langle \mathbf{K}, \mathbf{K} \rangle_{t}(\omega) 
\right]
=\frac{2}{t}\int_{0}^{t} \Q^{s}\| \nabla K +\nabla u_{\kk} \|^{2}(x, \xi) ds. 
\end{align*}
Due to the ergodicity of $\hh$, for $\hh$-a.e. $(x, \xi)$,
\begin{align}
\label{QV1}
\lim_{t\to\infty}
\mathbb{E}_{(x, \xi)}
\left[
\frac{1}{t}
\langle \mathbf{B}, \mathbf{B} \rangle_{t}(\omega)
\right]
&=
2\int 
\|
\nabla B + \nabla u_{\bb}
\|^{2}
d\hh,
\\
\label{QV2}
\lim_{t\to\infty}
\mathbb{E}_{(x, \xi)}
\left[
\frac{1}{t}
\langle \mathbf{K}, \mathbf{K} \rangle_{t}(\omega) 
\right]
&=
2\int 
\|
\nabla K + \nabla u_{\kk}
\|^{2} 
d\hh. 
\end{align}
Using Markov property, we have 
\begin{align*}
\mathbb{E}_{(x, \xi)}\left[ \frac{1}{t+1} \langle \mathbf{M}, \mathbf{M} \rangle_{t+1} \right]
&=\mathbb{E}_{(x, \xi)}
\left[ 
\frac{t}{t+1} \mathbb{E}_{\omega_{1}} \left[ \frac{1}{t} \langle \mathbf{M}, \mathbf{M} \rangle_{t} \right] 
\right]\\
&=\frac{t}{t+1} \mathbb{E}_{(x, \xi)} \left[ \frac{1}{t} \int_{0}^{t}\Q^{r}F (\omega) dr \right].
\end{align*}
for $\mathbf{M} = \mathbf{B}$ or $\mathbf{K}$ and $F = 2\| \nabla B +\nabla u_{\bb} \|^{2}$ or $2\| \nabla K +\nabla u_{\kk} \|^{2}$, respectively.
Given $x\in\MM$, for $\nu_{x}$-a.e. $\xi$ and $\mathbb{P}_{(x, \xi)}$-a.e. $\omega$,
\begin{equation*}
\lim_{t\to\infty} \frac{1}{t}\int_{0}^{t}\Q^{r}F(\omega) dr = \int F d\hh.
\end{equation*}
Hence for each $x$, there is $\xi$ for which we have the limits (\ref{QV1}) and (\ref{QV2}).
We denote the square root of the limits by $\sigma_{\bb}$ and $\sigma_{\kk}$, respectively.
Note that both of $\sigma_{\bb}, \sigma_{\kk}$ are positive 
since $B$ and $K$ are unbounded 
while $u_{\bb}$ and $u_{\kk}$ are bounded.
We have $\sigma_{\bb}, \sigma_{\kk} <\infty $ since both of $2\| \nabla B +\nabla u_{\bb} \|^{2}$ or $2\| \nabla K +\nabla u_{\kk} \|^{2}$ are bounded.
Thus for every $x$, there is $\xi$ such that the distributions of $\frac{\mathbf{B}_{t} }{\sigma_{\bb}\sqrt{t}}$ and $\frac{\mathbf{K}_{t}}{\sigma_{\kk}\sqrt{t}}$ under $\mathbb{P}_{(x, \xi)}$ converge to $N(0,1)$ as $t\to\infty$ due to the following lemma :

\begin{lem} (\cite{Hel_1982})
Let $(M_{t})_{0\le t\le\infty}$ be a continuous, centered, square-integrable martingale on a filtered probability space with stationary increments. 
If $M_{0}=0$ and there is $\sigma>0$ such that $\lim_{t\to\infty}\mathbb{E}[ |\frac{1}{t}\langle M, M\rangle_{t}-\sigma^{2}| ] = 0$, 
then the distribution of $\frac{1}{\sigma\sqrt{t}}M_{t}$ is asymptotically normal.
\end{lem}

Let $W^{\ld}_{t}(\omega):=\dist (\widetilde\omega_{0}, \widetilde\omega_{t})-t\ld$.
Since the distribution of $W^{\ld}_{t}$ under $\mathbb{P}_{(x, \xi)}$ and the distribution of $Y^{\ld}_{t}$ under $\mathbb{P}_{x}$ coincide, 
it is enough to show that $W^{\ld}_{t}$ and $\mathbf{B}_{t}$ have the same $\mathbb{P}_{(x, \xi)}$-distribution.
For  $\mathbb{P}_{(x, \xi)}$-a.e. 
$\omega$ and a lift $(\widetilde\omega, \xi)$, since 
$B(\omega_{t})-B(\omega_{0}) - \dist(\widetilde\omega_{0}, \widetilde\omega_{t})
=
\bb(\widetilde\omega_{t}, \widetilde\omega_{0}, \xi)
-
\dist(\widetilde\omega_{0}, \widetilde\omega_{t}) \to -2(\xi|\widetilde\omega_{\infty})_{\widetilde\omega_{0}}$
and $|(\xi|\widetilde\omega_{\infty})_{\widetilde\omega_{0}}|<\infty$,
\begin{equation*}
\lim_{t\to\infty} 
\frac{1}{\sigma_{\bb}\sqrt{t}}
\left[
B(\omega_{t})-B(\omega_{0})-\dist(\widetilde\omega_{0}, \widetilde\omega_{t})
\right]
=0.
\end{equation*}
Hence the distribution of 
$\frac{1}{\sigma_{\bb}\sqrt{t}}W^{\ld}_{t}$
under $\mathbb{P}_{(x, \xi)}$
also converges to the normal distribution since
\begin{align*}
W^{\ld}_{t}(\omega)
=
&
\left[
\dist(\widetilde\omega_{0}, \widetilde\omega_{t}) 
-B(\omega_{t})+B(\omega_{0})
\right]
-
\left[
u_{\bb}(\omega_{t})
-u_{\bb}(\omega_{0})
\right]
+
\mathbf{B}_{t}(\omega),
\end{align*}
and
\begin{equation*}
\frac{1}{\sigma_{\bb}\sqrt{t}}
\left|
u_{\bb}(\omega_{t})-u_{\bb}(\omega_{0})
\right|
\le
\frac{2}{\sigma_{\bb}\sqrt{t}}\|u_{\bb}\|_{\infty}
\to
0,
\textrm{ as }t\to\infty. 
\end{equation*}

Let $W^{\hs}_{t}(\omega):= \log \G(\widetilde\omega_{0}, \widetilde\omega_{t}) + t\hs$.
Since the $\mathbb{P}_{(x, \xi)}$-distribution of $W^{\hs}_{t}$ and the $\mathbb{P}_{x}$-distribution of $Y^{\hs}_{t}$ are the same, to verify that 
$\frac{1}{\sigma_{\kk}\sqrt{t}}
W^{\hs}_{t}$ 
is asymptotically normal,
it is sufficient to show that for $\mathbb{P}_{(x, \xi)}$-a.e. $\omega$ with a lift $(\widetilde\omega, \xi)$ to $\T^{1}\MM$, 
\begin{equation} \label{GMeq}
\limsup_{t\to\infty} 
\left|
\log \G(\widetilde\omega_{0}, \widetilde\omega_{t}) -K(\omega_{t})+K(\omega_{0})
\right|
<\infty.
\end{equation}
Note that for $\mathbb{P}_{(x, \xi)}$-a.e. $\omega$ with a lift $(\widetilde\omega, \xi)$, $K(\omega_{t})-K(\omega_{0})=\log \kk(\widetilde\omega_{0}, \widetilde\omega_{t}, \xi)$ and $\widetilde\omega_{\infty}\ne \xi$.
We denote by $z_{t}$ the closest point to $\widetilde\omega_{0}$ on the geodesic ray $[\widetilde\omega_{t}, \xi)$ generated by $(\widetilde\omega_{t}, \xi)$, 
$z_{t}$ converges to a point $z_{\infty}\in\MM$ on the geodesic $(\widetilde\omega_{\infty}, \xi)$ joining two boundary points $\widetilde\omega_{\infty}$ and $\xi$.
We have that for every $y$ on $[\widetilde\omega_{t}, \xi)$,
\begin{align*}
|\log\G(\widetilde\omega_{0}, \widetilde\omega_{t})-\log \kk(\widetilde\omega_{0}, \widetilde\omega_{t},\xi)|
\le
\left|\log \frac{\G(\widetilde\omega_{0},\widetilde\omega_{t})}{\G(z_{t}, \widetilde\omega_{t})}\right|
+
\left|
\log \frac{\G(z_{t}, \widetilde\omega_{t})}
{\left(\frac{\G(y, \widetilde\omega_{t})}{\G(y, z_{t})}\right)}
\right|
+
\left|\log\frac
{\left(\frac{\G(y, \widetilde\omega_{t})}{\G(y, z_{t})}\right)}
{\kk(\widetilde\omega_{0}, \widetilde\omega_{t},\xi)}\right|.
\end{align*}
Applying the Harnack inequality to the first term in the right handed side, 
since $\{\dist(\widetilde\omega_{0}, z_{t})\}_{t\ge0}$ is bounded,
it follows that 
$\left|\log\frac{\G(\widetilde\omega_{0}, \widetilde\omega_{t})}{\G(z_{t},\widetilde\omega_{t})}\right|
\le 
C_{1}$
for some constant $C_{1}=C_{1}(\widetilde\omega)>0$ dependent of $\widetilde\omega$ but not $t$.
And by the Ancona inequality (\cite{Ancona_1987}), 
the second term in the right handed side is also bounded by 
$C_{2}(\widetilde\omega)$. Letting $y$ tend to $\xi$, 
we see that the last term converges to
$\left|\log\frac{\kk(z_{t}, \widetilde\omega_{t}, \xi)}{\kk(\widetilde\omega_{0}, \widetilde\omega_{t}, \xi)}\right|$ which is also bounded by $C_{1}(\widetilde\omega)$ due to the Harnack inequality.
Therefore we have (\ref{GMeq}) and this completes the proof of Theorem \ref{CLT}.

\section{Proof of Theorem \ref{Ctr}}\label{ProofCtr}

In this section, we prove the contraction property on H\"older spaces of the foliated Brownian motion.
For the H\"older semi-norm, 
we prove a lower bound of the expectation of the Busemann functions at Brownian points which depends only on the dimension and the curvature bounds and linearly on time $T$.
The lower bound follows from the fact that the Laplacian of the Busemann function has the same lower bound with the Laplacian of the distance function due to the Rauch comparison theorem.
We also show the Doeblin property of the Brownian motion for the estimate of the uniform norm.

\begin{prop} \label{CtrH}
For sufficiently small $\tau$, there exists $C_{1}>0$ such that for each $t>0$,
\begin{equation*}
\sup_{x\in\M_{0}}
\sup_{\xi,\eta\in\partial\MM}
\frac{|\mathcal{Q}^{t}f(x,\xi)-\mathcal{Q}^{t}f(x,\eta)|}{d_{\infty}^{x,\tau}(\xi, \eta)}
\le
\|f\|_{\mathcal{L}^{\tau}} e^{-C_{1}t}.
\end{equation*}
\end{prop}
\begin{proof}
Since we have that
\begin{align*}
\frac{|\mathcal{Q}^{t}f(x,\xi)-\mathcal{Q}^{t}f(x,\eta)|}{d_{\infty}^{x,\tau}(\xi, \eta)}
&<\int_{\MM} 
\wp(t,x, y) 
\frac{\left|\widetilde{f}(y, \xi)-\widetilde{f}(y, \eta)\right|}{d_{\infty}^{x,\tau}(\xi, \eta)}
d\vol_{\MM}(y)
\\
&<\|f\|_{\mathcal{L}^{\tau}}
\int_{\MM}
\wp(t,x,y)
\frac{d_{\infty}^{y,\tau}(\xi, \eta)}{d_{\infty}^{x,\tau}(\xi, \eta)}
d\vol_{\MM}(y)\\
&=\|f\|_{\mathcal{L}^{\tau}}
\mathbb{E}_{x}
\left[
\frac{d_{\infty}^{y,\tau}(\xi, \eta)}{d_{\infty}^{x,\tau}(\xi, \eta)}
\right],
\end{align*}
it is sufficient to find $C_{1}>0$ such that
\begin{equation*}
\sup_{x, \xi, \eta}
\mathbb{E}_{x}
\left[
\frac{d_{\infty}^{\widetilde\omega_{t},\tau}(\xi, \eta)}{d_{\infty}^{x,\tau}(\xi, \eta)}
\right]
<
e^{-C_{1}t}.
\end{equation*}
Due to the Markov property of the Brownian motion,
\begin{align*}
\sup_{x, \xi, \eta}
\mathbb{E}_{x}
\left[
\frac{d_{\infty}^{\widetilde\omega_{t+s},\tau}(\xi, \eta)}{d_{\infty}^{x,\tau}(\xi, \eta)}
\right]
&=
\sup_{x, \xi, \eta}
\mathbb{E}_{x}
\left[
\frac{d_{\infty}^{\widetilde\omega_{s},\tau}(\xi, \eta)}{d_{\infty}^{x,\tau}(\xi, \eta)}
\mathbb{E}_{x}
\left[
\left.
\frac{d_{\infty}^{\widetilde\omega_{t+s},\tau}(\xi, \eta)}{d_{\infty}^{\widetilde\omega_{s},\tau}(\xi, \eta)}
\right| \mathscr{F}_{s}(\MM)
\right]
\right]\\
&\le
\sup_{x, \xi, \eta}
\mathbb{E}_{x}
\left[
\frac{d_{\infty}^{\widetilde\omega_{t},\tau}(\xi, \eta)}{d_{\infty}^{x,\tau}(\xi, \eta)}
\right]
\sup_{x, \xi, \eta}
\mathbb{E}_{x}
\left[
\frac{d_{\infty}^{\widetilde\omega_{s},\tau}(\xi, \eta)}{d_{\infty}^{x,\tau}(\xi, \eta)}
\right].
\end{align*}
Let us write $g(\widetilde\omega_{t}):=(\xi|\eta)_{\widetilde\omega_{t}} -(\xi|\eta)_{x}$.
Applying the Taylor theorem to the function $R\mapsto\exp(-\tau R)$
and substituting $g(\widetilde\omega_{t})$ for $R$, we have
\begin{align*}
\frac{d_{\infty}^{\widetilde\omega_{t},\tau}(\xi, \eta)}{d_{\infty}^{x,\tau}(\xi, \eta)}
\le 1-\tau g(\widetilde\omega_{t})
+\tau^{2}\dist_{}(x, \widetilde\omega_{t})^{2}
e^{2\tau\dist_{}(x,\widetilde\omega_{t})}.
\end{align*}
By Proposition \ref{HCom}, for some constant $C_{1}'>0$,
\begin{equation}\label{SEM}
\sup_{x}
\mathbb{E}_{x}
\left[
\dist_{}(x, \widetilde\omega_{t})^{2}
e^{2\tau\dist_{}(x,\widetilde\omega_{t})}
\right]
< C_{1}'.
\end{equation}
Therefore, with (\ref{SEM}) and Lemma \ref{ctrlem} below, we have
\begin{equation*}
\sup_{0\le t \le T}
\sup_{x, \xi,\eta}
\mathbb{E}_{x}
\left[
\frac{d_{\infty}^{\widetilde\omega_{t}, \tau}(\xi, \eta)}
{d_{\infty}^{x,\tau}(\xi, \eta)}
\right]
\le 1- \tau (d-1)a + \tau^{2} C_{1}'.
\end{equation*}
Fix $T\ge1$ 
and 
sufficiently small
$\tau$ 
such that 
$1- \tau (d-1)a + \tau^{2} C_{1}'<1$.
For such small $\tau$, 
put $C_{1}=(1-a(d-1)\tau+C_{1}'\tau^{2})^{\frac{1}{T}}$ and the inequality follows.
\end{proof}

\begin{lem}\label{ctrlem}
For every $T\ge0$,
\begin{equation*}
\inf_{x \in\M_{0}}\inf_{\xi\ne\eta} 
\mathbb{E}_{x}[(\xi|\eta)_{\widetilde\omega_{T}}-(\xi|\eta)_{x}]
\ge(d-1)a T.
\end{equation*}
\end{lem}

\begin{proof}[Proof of Lemma \ref{ctrlem}]
Due to the equation 
$$(\xi|\eta)_{x}-(\xi|\eta)_{y}= \frac{1}{2} \bb(x, y, \xi)+\frac{1}{2}\bb(x,y,\eta),$$
it suffices to show that
\begin{equation*}
\mathbb{E}_{x}[\bb(\widetilde\omega_{T}, x, \eta)]
\ge (d-1)a T.
\end{equation*}

Choose 
$z_{n}\in\MM$ 
such that 
$z_{n}\to \xi$ as $n\to\infty$
 and write 
$$f_{n}(y):=\bb(y, x, z_{n})= \dist_{}(y, z_{n})-\dist_{}(x, z_{n}).$$ 
By the Rauch's comparison theorem (see \cite{Petersen_2016}, for instance), 
\begin{align}\label{RCom}
\Delta f_{n} (y)
&= \Delta_{y} \dist_{}(y, z_{n})
\ge (d-1)\frac{ \textrm{sn}_{-a^2}'(\dist_{}(y, z_{n}))}{\textrm{sn}_{-a^2}(\dist_{}(y, z_{n}))}\\
&=a(d-1) \coth \left(a\,\dist_{}(y, z_{n})\right)
\end{align}
where $\textrm{sn}_{-a^{2}}(t)=\frac{1}{a}\sinh (at)$.

Let $f(y, \xi)=\bb(y, x,\xi)$. Then, since $\Delta_{s}$ is the generator of $\mathcal{Q}^{t}$ and $\Delta_{s}f(y,\xi) = \Delta_{y}f(y, \xi)$, 
$$
\mathbb{E}_{x}[\bb(\widetilde\omega_{T}, x, \xi)]
=\mathcal{Q}^{T}f(x, \xi)
=\int_{0}^{T} \mathcal{Q}^{t}\Delta_{s} f(x, \xi)dt
=\int_{0}^{T} \mathbb{E}_{x}[\Delta\bb(\widetilde\omega_{t}, x, \xi)]dt.
$$
Due to (\ref{RCom}) and Proposition \ref{C2conv},
\begin{equation*}
\mathbb{E}_{x}[\bb(\widetilde\omega_{T}, x, \xi)]\ge (d-1)aT
\end{equation*}
for every $x\in\MM$ and every $\xi\in\partial\MM$.
\end{proof}

Write $P(t,x,y)=\sum_{\gamma\in\Gamma} \wp(t, x, \gamma y)$ for $x, y\in\M_{0}$.
We have $\lim_{t\to\infty}P(t,x,y)=\frac{1}{\vol(\M)}$, in particular, $P(t, x, x)$ decreases as $t\to\infty$ (see \cite{ChaK_1991}).
We also have that
\begin{align}
\label{CSineq}
\left(
\int_{\M} 
\left|
P(t, x, y) -\frac{1}{\vol(\M)}
\right|
d\vol(y)
\right)^{2}
&\le
\vol(\M)
\int_{\M} 
\left|
P(t, x, y) -\frac{1}{\vol(\M)}
\right|^{2}
d\vol(y)\\
&=
\vol(\M)
\left(
P(2t, x, x)-\frac{1}{\vol(\M)}
\right).
\end{align}
Hence the integral on the left-handed side decreases to zero as $t$ goes to infinity. 
Indeed, it decays exponentially fast (see \cite{Don_1987}).
The following lemma shows that it has uniform exponential decay rate.

\begin{lem} \label{Doeblin}
There exists a constant $C_{2}=C_{2}(d, b)>0$ such that for each $x\in\M$,
\begin{equation*}
\int_{\M} 
\left|
P(t, x, y) -\frac{1}{\vol(\M)}
\right|
d\vol(y)
\le
C_{2}
e^{-\frac{\lambda_{1}}{2}t},
\end{equation*}
where $\lambda_{1}=\inf\{ \lambda>0: \lambda \in \mathrm{Spec}(\Delta_{\M})\}$.
\end{lem}

\begin{rmk} 
Since the bottom of the ($L^{2}$-)esssential spectrum $\lambda_{ess}:=\inf \mathrm{Spec}_{ess}\left(\Delta_{\M}\right)$ of the Laplacian is positive (\cite{Dodziuk_1987}) and $\mathrm{Spec} \left(\Delta_{\M}\right)\cap [0, \lambda_{ess})$ is discrete (\cite{Don_1987}), the smallest nonzero the spectrum $\lambda_{1}$ is also positive.
\end{rmk}

\begin{proof}
If we consider $\PP^{t}f(x):=\int (P(t, x, y)-\vol(\M)^{-1})f(y)d\vol(y)$ as a self-adjoint operator acting on the space $L^{2}_{0}(\M)$ of square-integrable functions with zero integral, 
$\Delta|_{L^{2}_{0}(\M)}$ is the generator of $P^t$ with the bottom of the spectrum $\lambda_{1}$.
Therefore the operator norm satisfies 
\begin{equation}\label{heatctr}
\|\PP^{t}\|\le e^{-\frac{\lambda_{1}t}{2}}
\end{equation}
for every $t>0$ (see the proof of Proposition V.1.2 in \cite{EngNag}).

For every $x\in\M_{0}$,
if we denote
$f_{t}(y) = P(t, x, y)-\frac{1}{\vol(\M)}$, then
$f_{t+t_{0}}(y) = \PP^{t}f_{t_{0}}(y)$.
It follows from (\ref{CSineq}) and (\ref{heatctr}) that
\begin{align*}
\int_{\M}
\left|
P(t+t_{0}, x, y) - \frac{1}{\vol(\M)}
\right|
d\vol(y)
&\le
\left(
\vol(\M)
\int_{\M} 
\left|
f_{t+t_{0}}(y)
\right|^{2}
d\vol(y)
\right)^{1/2}
\\
&\le
\|\PP^{t}\|
\left\|
f_{t_{0}}
\right\|_{2}
\le
e^{-\frac{\lambda_{1}t}{2}}
|f_{2t_{0}}(x)|^{1/2}
\\
&=
e^{-\frac{\lambda_{1}t}{2}}
\left|
P(2t_{0}, x,x)-\frac{1}{\vol(\M)}
\right|^{1/2}.
\end{align*}
Thus it suffices to prove that the diagonal supremum $\sup_{x\in \M} P(2 t_{0}, x, x)$ 
of the heat kernel on $\M$ is finite for some $t_{0}>0$.

Recall that we identify the fundamental domain $\M_{0}$ with $\M$ and 
$P(t,x, y)= \sum_{\gamma\in\Gamma} \wp(t, x, \gamma y)$.
In order to estimate the diagonal supremum $\sup_{x\in\M} P(t, x, x)$ of the heat kernel on $\M$, 
we shall use the Gaussian upper bound of the heat kernel on $\MM$ (Corollary 5 in \cite{Grigoryan_1994}): 
there is a constant $C=C(d, b)$ such that for each $t>1$,
\begin{equation}\label{Gaussbdd}
\wp(t,x,y)
\le 
C
\left(\frac{\dist_{}(x,y)^{2}}{t}\right)^{1+\frac{d}{2}}
\exp \left(-\frac{\dist_{}(x,y)^{2}}{4t}-\lambda_{0} t \right).
\end{equation}

Fix $x_{0}\in\M_{0}$.
For a cuspidal point $\xi \in \Pi(\M_{0}):=\partial\MM\cap \overline\M_{0}$,
we denote the cuspidal region of level $n$ based at $\xi$ by 
\begin{equation*}
\HH(\xi, n):= \{y \in \M_{0}: \bb(x_{0}, y, \xi)\ge n\}.
\end{equation*}
Let $x_{n}=x^{\xi}_n \in \HH(\xi, n)$ be the point in the geodesic ray joining $x_{0}$ and $\xi$ with $\bb(x_{0}, x_{n}, \xi)= n$. 
If $\gamma$ is in the stabilizer $\Gamma_{\xi}$ of $\xi$, then $x_{0}$ and $\gamma x_{0}$ are in the horosphere of the same level based at $\xi$. This implies that for every $\gamma \in \Gamma_{\xi}$,
\begin{equation} 
\label{Rc}
e^{-b(n+1)}\dist_{}(x_{0}, \gamma x_{0})
\le 
\dist_{}(x_{n}, \gamma x_{n})
\le 
e^{-an} \dist_{}(x_{0}, \gamma x_{0}).
\end{equation} 
Applying (\ref{Rc}) to the Gaussian bound (\ref{Gaussbdd}), for each $\gamma \in \Gamma_{\xi}$,
\begin{equation}\label{Parabolic}
\wp(t, x_{n}, \gamma x_{n})
\le 
C
e^{-\lambda_{0} t}
\dist_{}(x_{0},\gamma x_{0})^{d+2}
\exp
\left(
-\frac{\dist_{}(x_{0}, \gamma x_{0})^{2}}{4t e^{2b(d-1)(n+1)}}
-a(d+2)n
\right).
\end{equation}
We want to show that given $\delta>0$, 
there is $t>0$ such that for every sufficiently large $n$,
\begin{equation*}
\textrm{the right-hand side of (\ref{Parabolic}) }
\le
e^{-\delta \dist(x_{0}, \gamma x_{0})}.
\end{equation*}

To simplify the notation, we put
\begin{equation*}
f_{n,\xi}(R)
:= 
R^{d+2}
\exp
\left(
-\frac{R^{2}}{4t e^{ 2b(d-1)(n+1)}}
+\delta R
\right).
\end{equation*}
Since its derivative is
\begin{equation*}
f_{n, \xi}'(R)= 
R^{d+1} 
\left(
d+2 - \frac{R^{2}}{2t} e^{-2b(d+2)(n+1)} +\delta R
\right)
\exp
\left( -\frac{R^{2}}{4t} e^{-2b(d+2)(n+1)}+\delta R\right), 
\end{equation*}
the positive nonzero extreme point of $f_{n,\xi}$ is 
$R_{n}:=t\delta e^{2b(n+1)}+\sqrt{t^{2}\delta^{2}e^{4b(n+1)}+2t(d+2)e^{2b(n+1)}}$.
Thus $f_{n, \xi}$ has the maximum on $\mathbb{R}_{+}$ at $R_{n}$:
\begin{align*}
f_{n, \xi}(R)
&\le 
f_{n,\xi}(R_{n})\\
&=
R_{n}^{d+2}
\exp\left(-\frac{R_{n}^{2}}{4t} +\delta R_{n} -a(d+2)n\right)
\\
&
\le
\left(3t\delta e^{2b(n+1)}\right)^{d+2}
\exp
\left(
-\frac{t\delta^{2} e^{4b(n+1)}}{2}
+3t\delta^{2} e^{2b(n+1)}
-a(d+2)n
\right)
\\
&
=
\left[
(3t\delta)
e^{2b(n+1)-an}
\right]^{d+2}
e^{-\frac{9\delta^{2}t}{2}}
\exp
\left(
-\frac{\delta^{2}t}{2}
( e^{2b(n+1)} -3)^{2}
\right).
\end{align*}
Therefore, there is $N_{\xi}(\delta, t)$ such that if $n>N_{\xi}(\delta, t)$, then
$f_{n,\xi}(R)\le C^{-1} t^{-1-\frac{d}{2}}e^{\lambda_{0} t}$, 
hence
$\wp(t, x_{n}, \gamma x_{n}) \le e^{-\delta \dist( x_{0}, \gamma x_{0})}$.
We conclude that
\begin{align}\label{cuspidal}
\sum_{\gamma\in \Gamma_{\xi}}
\wp(t, x_{n}, \gamma x_{n})
\le
\sum_{\gamma\in \Gamma_{\xi}}
e^{-\delta \dist(x_{0}, \gamma x_{0})} = Q_{\Gamma_{\xi}, x_{0}}(\delta),
\end{align}
where $Q_{G, x}(\delta):=\sum_{g\in G}e^{-\delta \dist(x, g x)}$ denotes the Poincar\'e series of a discrete group $G$ of isometries on $\MM$.
We denote the abscissa of convergence of $Q_{G, x}$, which is called the \textit{critical exponent} of $G$, by $\delta_{G}$.

Put $N_{\xi}:= N_{\xi}\left(\delta_{\Gamma}+1, t\right)$ 
and choose $N$ larger than $\max_{\xi \in \Pi(\M_{0})} N_{\xi}$. We define a truncated domain in the fundamental domain $\M_{0}$ by
\begin{equation*}
\M_{N}:= \M_{0} \setminus \bigcup_{\xi\in\Pi(\M_{0})} \HH\left(\xi,  N \right).
\end{equation*}
Note that $\M_{N}$ is a pre-compact domain. 
Take $x_{0}\in\M_{N}$ and $x \in \HH\left(\xi, N\right)$ for some $\xi\in \Pi(\M_{0})$.
Then we can replace $x$ by $x_{n}=x^{\xi}_{n}$ for some $n\ge N$:
there is $n\ge N$ such that $x\in \HH(\xi, n) \setminus \HH(\xi, n+1)$ and $d(x, x_{n})$ is bounded uniformly on $n\ge N$.

We may assume that given $t>0$, 
$g(R)= \left(\frac{R^{2}}{t}\right)^{1+\frac{d}{2}}\exp\left(-\frac{R^2}{4t}\right)$ is decreasing and
$g(R)\le \exp(-(\delta_{\Gamma}+1)R)$ for every $R>0$.
Assume that $x_{n}$ is on the geodesic ray $[x_{0}, \xi)$ joining $x_{0}$ and $\xi$ and $d(x_{0}, x_{n})=n$.
From $\dist(x_{N}, \gamma x_{N}) -2(n-N) \le \dist (x_{n}, \gamma x_{n})$, writing $R_{n}=\dist(x_{n}, \gamma x_{n})$ for $n\ge N$, there exists $C'=C'(d)>1$ such that
\begin{align*}
g( R_{n} )
&\le g( R_{N} -2(n-N) )\\
&= \left(\frac{(R_{N}-2(n-N))^{2}}{t}\right)^{1+d/2} \exp\left(-\frac{(R_{N}-2(n-N))^{2}}{4t}\right)\\
&\le C' g(R_{N}) g_{N}( 2(n-N) ),
\end{align*}
where $g_{N}(T):= \left(\frac{T}{\sqrt{t}}\right)^{d+2}\exp\left(-\frac{T^2 - R_{N}T }{4t} \right)$.
By the similar computation as in (\ref{cuspidal}),
\begin{equation*}
g_{N}(T) 
\le g_{N} (T_{N}) 
\le \left(\frac{R_{N}}{\sqrt{t}}\right)^{d+2} \exp\left(\frac{15}{64t}R_{N}^{2}\right),
\end{equation*} 
where $T_{N}$ the critical value of $g_{N}$.
Thus we have
\begin{equation*}
g(R_{n}) 
\le C' \left(\frac{R_{N}}{t}\right)^{2d+4} \exp\left( -\frac{1}{16t}R_{N}^2 \right) 
\le C'' e^{-(\delta_{\Gamma}+1)\dist(x_{N}, \gamma x_{N})},
\end{equation*}
for some $C''>1$ independent of $N$.
Then it follows that 
\begin{align*}
P(t, x, x)
\le 
&
C
\sum_{\gamma\in\Gamma} 
\left(
\frac{\dist_{}(x, \gamma x)^{2}}{t}
\right)^{1+\frac{d}{2}}
\exp
\left(
-\frac{\dist(x,\gamma x)^{2}}{4t}
-\lambda t 
\right)\\
\le
&
C
Q_{\Gamma_{\xi}, x_{0}}(\delta_{\Gamma}+1)
+
CC''
\sum_{\gamma\notin\Gamma_{\xi}} 
e^{-(\delta_{\Gamma}+1)\dist(x_{N}, \gamma x_{N})}\\
\le
&
C(1+C'') \max
\left\{Q_{\Gamma, x_{0}}(\delta_{\Gamma}+1), Q_{\Gamma, x_{N}^{\xi}}(\delta_{\Gamma})\right\}.
\end{align*}

Hence we have $\sup_{x\in \HH(\xi, N)} P(t, x, x,)<\infty$ for every $\xi \in \Pi(\M_{0})$.
Therefore, since $\Pi(\M_{0})$ is a finite set, $\sup_{x\in\M}P(t, x, x) <\infty$.
\end{proof}

We are ready to verify the exponential decay of uniform norm and complete the proof of Theorem \ref{Ctr}. 
It is enough to show that the exponential decay of the supremum norm 
since we have already proved the exponential decay of H\"older norm in Proposition \ref{CtrH}.
\begin{prop}
There exists a constant $C_{2}>0$ such that for every $f\in\LL$, $t>0$
\begin{equation*}
\| \Q^{t}f -\N f\|_{\infty} 
\le 
\|f\|_{\L_{\tau}}
e^{-C_{2}t}.
\end{equation*}
\end{prop}
\begin{proof} 
Denote $F_{t}(x):=\int \Q^{t}f(x,\xi) d\nu_{x}(\xi)$.
\begin{align*}
\left| \Q^{t}f(x,\xi) - \int f d\hh\right|
&=
\left| 
\Q_{t}f(x,\xi) -\int \Q^{\frac{t}{2}}f d\hh
\right|\\
&\le
\left| 
\Q^{t}f(x,\xi) - \Q^{\frac{t}{2}}F_{\frac{t}{2}}(x)
\right|
+
\left|
\Q^{\frac{t}{2}}F_{\frac{t}{2}}(x)-\int \Q^{\frac{t}{2}}f d\hh
\right|\\
&\le
\left| 
\Q^{\frac{t}{2}}\left(\Q^{\frac{t}{2}}f(x,\xi) - F_{\frac{t}{2}}(x)\right)
\right|
+
\left|
\Q^{\frac{t}{2}}F_{\frac{t}{2}}(x)-\int \Q^{\frac{t}{2}}f d\hh
\right|.
\end{align*}
By Lemma \ref{Doeblin}, the last term of the last inequality decays exponentially:
\begin{align*}
\left|
\Q^{\frac{t}{2}}F_{\frac{t}{2}}(x)
-\int \Q^{\frac{t}{2}}f d\hh
\right|
&=
\left|
\int_{\M_{0}} 
P(t/2, x, y) F_{\frac{t}{2}}(y) 
d \vol(y)
-
\int_{\M_{0}}
F_{\frac{t}{2}}(y)
d\widetilde{m}(y)
\right|
\\
&\le
\|F_{\frac{t}{2}}\|_{\infty}
\int_{\M_{0}}
\left|
P(t/2, x, y)
-
\frac{1}{\vol(\M_{0})}
\right|
d\vol(y)
\\
&\le
\|f\|_{\tau}
e^{-\frac{\lambda_{1}t}{4}}.
\end{align*}
For the first term, it follows from Proposition \ref{CtrH} that
\begin{align*}
\left|
\Q^{\frac{t}{2}}
\left(
\Q^{\frac{t}{2}}f(x,\xi)-F_{\frac{t}{2}}(x,\xi)
\right)
\right|
&\le 
\sup_{y\in \M_{0}}
\left|
\Q^{\frac{t}{2}}f(y,\xi)-F_{\frac{t}{2}}(y,\xi)
\right|\\
&\le
\sup_{y\in\M_{0}}
\int 
\left|
\Q^{\frac{t}{2}}f(y, \xi)-\Q^{\frac{t}{2}}f(y, \eta)
\right|
d \nu_{y}(\eta)\\
&\le
\|f\|_{\tau}e^{-\frac{C_{1}t}{2}}.
\end{align*}
\end{proof}

\section{Ergodic properties of Brownian motions}
In this section, we discuss the thermodynamic formalisms for the harmonic potential, which arises from the Brownian motion and an equidistribution theorem of Brownian paths. Using such ergodic properties of the Brownian motion, we also provide a characterization of the asymptotic harmonicity as an application of the central limit theorem to the ergodic theory of the geodesic flow on $\M$.

\subsection{Harmonic potentials}
We introduce another natural potential $\K$ on $\T^{1}\M$ 
induced from the Brownian motion, 
which we call the \textit{harmonic potential}. 
Define a function $\widetilde{\K}$ on $\T^{1}\MM$ by
\begin{equation*}
\widetilde{\K}(\rv)=-\left.\frac{d}{dt}\right|_{t=0} \log \kk(\gamma_\rv(0), \gamma_{\rv}(t), \rv_+)
\end{equation*} 
where $\rv_{+}$ denote the end point at infinity $\lim_{t\to\infty}\gamma_{\rv}(t)$ of the geodesic $\gamma_{\rv}$ generated by $\rv$ and $\kk(x,y, \xi)$ is the Martin kernel of the Brownian motion on $\MM$. 
Since $\widetilde\K$ is a $\Gamma$-invariant H\"older continuous function on $\T^{1}\MM$ (\cite{Hamenst_dt_1990}),
it induces a H\"older potential, which is denoted by $\K$, on $\T^{1}\M$.

Note that $\K$ has the harmonic measure $(\nu_{x})_{x\in\MM}$ as a Patterson-Sullivan density of dimension $0$. 
Since the harmonic measure does not have atom (\cite{KifLed_1990}, \cite{BenHul_2019}), the set $\Pi_{\Gamma}$ of parabolic fixed points on $\Gamma$ in $\partial\MM$ has countably many points, $\Pi_{\Gamma}$ is a null set for the harmonic measure. 
As the set of conical fixed points $\Lambda_{c}\Gamma=\partial\MM \setminus \Pi_{\Gamma}$ has positive measure with respect to the harmonic measure, the topological pressure of $\K$ vanishes; $P_{\K}=0$  (Corollary 5.10 of \cite{PPS}).
We denote by $\widetilde\nu$ the Gibbs measure on $\T^{1}\MM$ of $\K$ and $(\nu_{x})$. 
Proposition \ref{VP} for $\K$ demonstrates that $\K$ admits an equilibrium state $\nu$ on $\T^{1}\M$ for $\K$ if and only if $\widetilde\nu \left(\T^{1}\M_{0}\right)$ is finite and $\nu$ agrees with the induced measure on $\T^{1}\M$ by $\widetilde\nu$.
From Proposition \ref{SPR} it follows that $\K$ admits an equilibrium state if and only if for every parabolic subgroup $\Pi$ of $\Gamma$,
\begin{equation*}
\sum_{\gamma\in\Pi} 
\frac
{\dist_{}(x, \gamma x)}
{\kk(x, \gamma x, (\rv_{x}^{\gamma x})_{+})} <\infty,
\end{equation*}
where $\rv_{x}^{y}\in\T_{x}^{1}\MM$ such that $\g^{\dist(x,y)}\rv_{x}^{y}\in\T^{1}_{y}\MM$.
We shall provide dynamical aspects of Brownian motions using the ergodic theory of $\nu$.

Recall that given $x\in\MM$ we identify 
$(r, \rv)\in(0,\infty)\times\T^{1}_{x}\MM$ 
with $\exp_{x}(r\rv)\in\MM\setminus\{x\}$
and $g=dr^{2}+\lambda_{x}(r, \rv)g_{\mathbb{S}}$.
Now we denote the density of volume at $z=(r, \rv)$ with respect to the polar coordinate at $x$ by $A_{x}(z)$:
\begin{equation*}
d \vol (z) = A_{x}(z) dr d\vol_{\mathbb{S}}(\rv).
\end{equation*}
Note that $A_{x}(z)= \lambda_{x}^{d-1}(r, \rv)$.
We denote by $\theta(\widetilde\omega, t)$ 
the unit vector in $\T^{1}_{x}\MM$ such that 
$\widetilde\omega_{t}
=(r, \theta)(\widetilde\omega, t)
:=(r(\widetilde\omega, t), \theta(\widetilde\omega, t))$.

The following theorem demonstrates how dynamical invariants and stochastic invariants are related to each other. 
We follow the argument in \cite{Led_1988}, 
but we complete the proof by showing the inequality $\hs \le \ld h_{\nu}$ using the idea in \cite{Led_1987}

\begin{thm}\label{SE} 
If $\K$ admits an equilibrium state $\nu$, then
\begin{equation*}
\hs = \ld h_{\nu}.
\end{equation*}
\end{thm}

\begin{proof}
Let $x\in\MM$, $\delta\in \left( 0, \frac{1}{2} \right)$ and $0<\varepsilon, \varepsilon'$.
We denote for each $T>0$,
\begin{align*}
\mathscr{C}_{T} 
:=
\{\widetilde\omega: \
&
\dist(\widetilde\omega_{T}, (\ld T, \widetilde\omega_{\infty}))\le \varepsilon T
\textrm{ and }\\
&
\mu^{\T}_{x}
\{ \rv: 
d_{\ld T}(\rv, \theta(\widetilde\omega, \infty))\le \varepsilon'
\}
\le 
e^{-(\ld h_{\nu}-\varepsilon)T}
\},
\\
\mathscr{D}_{T}
:=
\{\widetilde\omega:
&
\dist(\widetilde\omega_{T}, (\ld T, \widetilde\omega_{\infty}))\le \varepsilon T
\textrm{ and }\\
&
\mu^{\T}_{x}
\{ \rv: 
\dist \left(\gamma_{\rv}(\ld T), \gamma_{\theta(\widetilde\omega, \infty)}(\ld T)\right)\le \varepsilon'
\}
\ge 
e^{-(\ld h_{\nu}+\varepsilon)T}
\}.
\end{align*}

For every $T$ large enough, 
$\mathbb{P}_{x}(\mathscr{C}_{T})\ge 2\delta$ for some $\varepsilon'>0$ 
by (\ref{roughpath}).
Thus if we fix a sufficiently large $T$ and 
choose $E\subset\MM$ with $\card E = N(x, T, 1-\delta)$, 
\begin{equation*}
\mathbb{P}_{x}\{ \dist(\widetilde\omega_{T}, E)\le 1\}\ge 1-\delta.
\end{equation*}
We note that 
$E_{\infty}:=
\{ \theta(\widetilde\omega,\infty):
\widetilde\omega\in \mathscr{C}_{T}, \dist( \widetilde\omega_{T}, E)\le 1
\}$
has the $\mu^{\T}_{x}$-measure greater than $\delta$ and
$\{\gamma_{\theta(\widetilde\omega, \infty)}(\ld T): 
\widetilde\omega\in \mathscr{C}_{T}, \dist( \widetilde\omega_{T}, E)\le 1
\}$
is covered by balls on the sphere of radius $\varepsilon'$ less than $N(x, T, 1-\delta) C^{\varepsilon T}$. 
($C$ is the maximal cardinal of covers for the intersection of the sphere of radius $\ld T$ and $(\varepsilon+1) T$ balls by $\varepsilon'$ balls on the sphere.) 
Such ball $O$ in the sphere of radius $\varepsilon'$ is the set of base points of vectors in 
$\g^{\ld T}V$ where $V=\{\rv: d_{\ld T} (\rv, \rw)\le \varepsilon'\}$ for some $\rw$.
We conclude that since such $V$ has the $\mu_{\T^{1}\M}$-measure less than $e^{-(\ld h_{\nu}-\varepsilon)T}$,
\begin{align*}
\delta
\le
\mu^{\T}_{x} (E_{\infty})
\le N(x, T, 1-\delta) e^{ -T[\ld h_{\nu} -\varepsilon-\varepsilon\log C]}.
\end{align*}
Thus we have $\ld h_{\nu} \le \lim_{T\to\infty} \frac{1}{T} \log N(x, T, 1-\delta)$.

Choose a smallest set $E\subset \MM$ such that $\dist(\widetilde\omega_{T}, E)\le 1$ for each 
$\widetilde\omega\in\mathscr{D}_{T}$ and
a maximal $\varepsilon'$-separeted set 
$F\subset \{\gamma_{\theta(\widetilde\omega, \infty)}(\ld T): \widetilde\omega \in \mathscr{D}_{T}\}$.
Since 
$\mathscr{D}_{T}\subset\{ \widetilde\omega: \dist(\widetilde\omega_{T}, E)\le 1\}$, 
$\card (E) \ge N(x, T, \mathbb{P}_{x}(\mathscr{D}_{T}))$
and 
$\card (F) \le C' e^{\ld h_{\nu}T+\varepsilon T}$.
($C'$ is the maximal number of overlappings.)
For every $f\in F$ if we denote 
\begin{equation*}
N(f):=\{ e \in E: 
\exists \, \widetilde\omega \in \mathscr{D}_{T} \textrm{ s.t. }
\dist(f, \gamma_{\theta(\widetilde\omega, \infty)}(\ld T))\le \varepsilon', 
\dist(e, \widetilde\omega_{T}) \le 1\}.
\end{equation*}
Since $\cup_{e\in N(f)} B(e, 1) \subset B(f, \varepsilon T +\varepsilon'+1)$,
there exists $C''>0$ such that
\begin{equation*}
\card{N(f)} 
\le \sup_{e\in E, f \in F} \frac{\vol\left( B(f, \varepsilon T+\varepsilon'+1)\right)}{\vol( B(e, 1))} 
\le e^{C''\varepsilon T}.
\end{equation*}
Therefore, 
\begin{equation*}
N(x, T, \mathbb{P}_{x}(\mathscr{D}_{T}) )\le \card (E) \le \exp(C''\varepsilon T)\card F\le e^{T[\ld h_{\nu} + (1+ C'')\varepsilon ]}.
\end{equation*}
\end{proof}

The following proposition means that Brownian paths are equidistributed with respect to $\nu$;
Geodesics which Brownian paths roughly follow are generic with respect to $\nu$.
The proof follows the argument for compact manifolds (\cite{Led_1988}).

\begin{prop}\label{eqd}
Assume that $\K$ admits an equilibrium state $\nu$.
For every
$x\in\MM$,
for each bounded continuous function
$\phi \in \C_{b}(\T^{1}\M)$
and for
$\mathbb{P}_{x}$-a.e. 
$\widetilde\omega$,
\begin{equation*}
\int \phi \, d\nu
=
\lim_{t\to\infty}
\frac{1}{\ld t}
\int_{0}^{r(\widetilde\omega, t)}
\widetilde{\phi}( \g^{s}\theta(\widetilde\omega, t))
\,ds.
\end{equation*}
\end{prop}

\begin{proof}
For $\rv, \rw\in \T^{1}\MM$, let $d_{t}(\rv, \rw)$ be the distance on the geodesic sphere $S(x, t)$ between $\g^{t} \rv$ and $\g^{t} \rw$. 
Then 
\begin{equation*}
d_{t}(\rv, \rw)\le d_{s}(\rv, \rw)\frac{\sinh(at)}{\sinh(as)}
\end{equation*} 
for every $0<t<s$ due to the curvature upper bound $\sec_{\MM}\le -a^{2}<0$.
Since the Sasaki distance is H\"older equivalent to the distance 
$\dist_{0}(\rv, \rw):= \sup_{0\le t\le1}\dist\left( \gamma_{\rv}(t), \gamma_{\rw}(t)\right)$, 
\begin{align*}
\left|
\int_{0}^{t} \widetilde\phi\circ\g^{s}(\rv) ds
-
\int_{0}^{t} \widetilde\phi\circ\g^{s}(\rw) ds
\right|
\le
C(a, \phi) \dist( \gamma_{\rv}(t), \gamma_{\rw}(t) ).
\end{align*}
Hence the proposition follows from of Proposition \ref{transversalerg} and the limit (\ref{roughpath}): 
for $\mathbb{P}_{x}$-a.e. $\widetilde\omega$,
\begin{align*}
&\lim_{t\to\infty}
\left| 
\frac{1}{\ld t}
\int_{0}^{r(\widetilde\omega, t)} 
\widetilde\phi(\g^{s}\theta(\widetilde\omega, t))
ds
-\int\phi d\nu
\right|\\
&\le
\lim_{t\to\infty}
\frac{1}{\ld t}
\left|
\int_{0}^{r(\widetilde\omega, t)}
\widetilde\phi(\g^{s}\theta(\widetilde\omega, t))
ds
-\int_{0}^{\ld t}
\widetilde\phi(\g^{s}\theta(\widetilde\omega, \infty))
ds
\right|
+
\lim_{t\to\infty}
\left|
\frac{1}{\ld t}
\int_{0}^{\ld t}
\widetilde\phi(\g^{s}\theta(\widetilde\omega, \infty))
ds
-
\int\phi d\nu
\right|\\
&=
\lim_{t\to\infty}
\frac{1}{\ld t}
\left|
\int_{0}^{r(\widetilde\omega, t)}
\widetilde\phi(\g^{s}\theta(\widetilde\omega, t))
ds
-
\int_{0}^{r(\widetilde\omega, t)}
\widetilde\phi(\g^{s}\theta(\widetilde\omega, \infty))
ds
\right|
+0
\\
&\le
\lim_{t\to\infty}
\frac{C(a, \phi)}{\ld t}
\dist (\widetilde\omega_{t}, (r, \theta)(\widetilde\omega, t))
=0.
\end{align*}
We used (\ref{transversalpterg}) of Proposition \ref{transversalerg} in the equation and (\ref{roughpath}) in the last inequality.
\end{proof}

The equidistribution of Brownian paths provides another stochastic invariant, the exponential growth along Brownian paths.
It helps understanding the relation between the harmonic measure class and the Lebesgue measure class.
The proof in \cite{Led_1988} extends to the finite-volume case.

\begin{thm} 
For each $x\in\MM$ and for $\mathbb{P}_{x,}$-a.e. $\widetilde\omega$, if $\K$ admits an equilibrium state $\nu$, the following limit exists:
\begin{equation*}
\Upsilon=\lim_{t\to\infty} \frac{1}{t}\log {A}(x, \widetilde\omega_{t})=-\ld \int \J d\nu.
\end{equation*}
\end{thm}

\begin{proof}
Let $\T_{\rv}\g^{t}$ be the tangent map of the flow map $\g^{t}$ at $\rv=(x, \xi) \in \T^{1}\MM$.
Since the angle between stable distribution 
$E^s(\rv)$ 
and 
$\T_{\rv}\g^{t}(\T_{\rv}\T^{1}_{x}\MM)$, 
where $\T_{\rv}\T^{1}_{x}\MM$ is the tangent space of the sphere $\T^{1}_{x}\MM$, 
is bounded away from zero uniformly on $\rv$ and $t>0$, 
\begin{align*}
\lim_{t\to\infty} \frac{1}{t} \log A(\pi\rv, \pi\g^{t}\rv) 
&=\lim_{t\to\infty} \frac{1}{t} \log \det \T\g^{t}|_{\T_{\rv}\T^{1}_{x}\MM}\\
&=\lim_{t\to\infty} \frac{1}{t} \log \det \T\g^{t}|_{E^{uu}(\rv)}\\
&=\lim_{t\to\infty}-\frac{1}{t}\int_{0}^{t} \J(\g^{s} \rv) ds.
\end{align*}

Therefore, by Proposition \ref{eqd}, 
for $\mathbb{P}_{x}$-a.e. $\widetilde\omega$, 
\begin{align*}
\lim_{t\to\infty} 
\frac{1}{t}
\log 
A(x, \widetilde\omega_{t})
=
\lim_{t\to\infty}
-\frac{1}{t}
\int_{0}^{r(\widetilde\omega, t)}
\J(\g^{s} \theta(\widetilde\omega_{t}))
ds
= -\ld \int \J d\nu.
\end{align*}
\end{proof}

Since $h_{\nu} \le \htop$ and $h_{\nu}+\int \J d\nu \le P_{\J}=0$,
from the previous theorems, we have the following theorem as a corollary.

\begin{thm}
Denote the topological entropy of $(\T^{1}\M, (\g^{t}))$ by $\htop$.
\begin{enumerate}
\item
$\hs \le \ld \htop$.
The equality holds if and only if the harmonic measure class and the visibility class coincide.
\item
$\hs \le \Upsilon = -\ld \int \J \,d\nu$.
The equality holds if and only if the harmonic measure class and the Lebesgue class agree.
\end{enumerate}
\end{thm}

\begin{proof}
Due to Theorem \ref{SE}, the equality 
$\Upsilon=\hs$ 
is equivalant to
\begin{equation*}
P(\J, \nu)=h_{\nu}+\int \J d\nu=0,
\end{equation*}
which holds if and only if $\nu$ is the equilibrium state for $\J$. 
\end{proof}

\subsection{Proof of Theorem \ref{AsympH}}
We conclude this section with the proof of Theorem \ref{AsympH}. 
We begin with the proof of the integral equation for the foliated Laplacian (\cite{Yue_1991}): 
for every bounded function $\varphi$ uniformly $\C^{2}$ on stable leaves,
\begin{equation} \label{yue}
\int
2
\langle 
\nabla \log \kk, 
\nabla \varphi
\rangle d\hh
= -\int \Delta_{s} \varphi d\hh.
\end{equation}
Consider the function 
$\Phi(y)
:= 
\int_{\partial\MM} \varphi(y, \xi) d\nu_{y}(\xi)
=
\int_{\partial\MM} \varphi(y, \xi) \kk(x, y, \xi) d\nu_{x}(\xi)$. 
Applying the Laplacian, 
since $\Delta_{y} \kk(x, y, \xi)=0$ we have 
\begin{align*}
\Delta \Phi (y)
&= 
\int_{\partial\MM} 
\kk(x, y, \xi) \Delta_{s} \varphi (y, \xi)
+
2
\langle
\nabla_{y} \varphi(y, \xi) , \nabla_{y} \kk(x, y, \xi)
\rangle
d\nu_{x}(\xi)\\
&=
\int_{\partial\MM}
\Delta_{s}\varphi(y, \xi)
+
2
\langle
\nabla_{y} \varphi(y, \xi), \nabla_{y} \log \kk (x, y, \xi)
\rangle
d\nu_{y}(\xi).
\end{align*}
Thus integrating with respect to $\vol$ and using Green's formula, since $\Phi$ is uniformly $\C^{2}$, 
\begin{align*}
&\int_{\T^{1}\M} 
\Delta_{s}\varphi (y, \xi)
+ 
2 \langle
\nabla_{y} \varphi(y, \xi) , \nabla_{y} \log \kk(x, y, \xi)
\rangle
d\hh (y, \xi)\\
&=
\int_{\M} \Delta\Phi(x) d\vol(x)\\
&=
\lim_{\varepsilon \to 0} \int_{\M_{\varepsilon}}\Delta\Phi(x) d\vol(x)\\
&=\lim_{\varepsilon \to 0} \int_{\partial \M_{\varepsilon}} \langle \nabla \Phi , \textbf{n}_{\varepsilon}\rangle
=0,
\end{align*}
where $\M_{\varepsilon}=\{ x\in \M: inj(x) \ge \varepsilon\}$ and $\textbf{n}_{\varepsilon}$ is the unit normal vector on $\partial \M_{\epsilon}$.

From the integral formula (\ref{yue}) for the foliated Laplacian, it follows that
\begin{equation*}
\sigma_{\kk}^{2}
=
2\int_{\T^{1}\M} \|\nabla \log \kk(x, \cdot, \xi) + \nabla u_{\kk}\|^{2} d\hh
=
2\int \| \nabla \log \kk\|^{2}+\| \nabla u_{\kk}\|^{2} d\hh,
\end{equation*}
since $\int \Delta_{s} u_{\kk} d\hh=0$.
Since $\int \| \nabla \log \kk \|^{2} d\hh= \hs$ (Proposition \ref{intcha}),
\begin{equation*}
\sigma_{\kk}^{2}-2\hs=2\int \| \nabla u_{\kk}\|^{2} d\hh \ge 0, 
\end{equation*}
and the equality holds if and only if $u_{\kk}$ is constant.
Since $u_{\kk}$ is the solution of the leafwise heat equation $\Delta_{s}u=\|\nabla\log \kk\|^{2}-\hs$,
$u_{\kk}$ is constant if and only if $\|\nabla \log \kk\|$ is constant and equal to $\hs$.

First we verify that $\|\nabla\log\kk\|^{2}=\hs$ implies the asymptotic harmonicity of $\M$. 
\begin{align*}
\hs 
&\le \ld \htop
\le \sup_{\mu}-\ld \int \K d\mu \\
&=\sup_{\mu} \ld \int \langle X, \nabla\log\kk \rangle d\mu 
\le \sup_{\mu} \ld \left|\int \|\nabla\log\kk\|^{2} d\mu\right|^{1/2}
= \ld\sqrt{\hs}\le \hs,
\end{align*}
where $X(x, \xi) := (x, \xi)$ and the supremum taken among invariant measures. 
Hence we have equalities and $\nu$ is the measure of maximal entropy. 
Replacing the supremum of integrations by integrations with respect to $\nu$, we have 
\begin{equation*}
\int \langle X, \nabla\log\kk \rangle d\mu = \left|\int \|\nabla\log\kk\|^{2} d\mu\right|^{1/2},
\end{equation*} 
which occurs if and only if $X=\nabla\log\kk$. 
Therefore the mean curvature of the stable horosphere $\mathrm{div}X=\Delta \log \kk = -\|\nabla \log \kk\|^{2}= -h$ is constant.

Conversely, if $\M$ is asymptotically harmonic, the geometric potential $\J=\mathrm{div}X$ is constant.
Since $P_{\J}=0$, the Liouville measure is the measure of maximal entropy and $\J=\htop$.
For $x, y\in\MM$ and $\xi\in\partial\MM$, if we write 
\begin{equation*}
\Psi_{x,\xi}(y) = \exp (-\htop\bb(y, x, \xi)),
\end{equation*} 
then we have $\Delta \Psi_{x,\xi}=0$ and $\lim_{y\to \eta} \Psi_{x, \xi}=0$ if $\eta \ne \xi$ and $\infty$ if $\eta=\xi$.
Hence $\kk(x, y, \xi) = \Psi_{x,\xi}(y)$ and $\| \nabla \log \kk\|^{2}=\htop^2$.
It follows that the harmonic class and the visibility class coincide, which implies $\hs =\ld \htop$.
Since $\ld = -\int \mathrm{div} X dm^{\Q}=\htop$, $\hs= \htop^{2}= \|\nabla \log \kk\|^{2}$. 
Therefore we have $\sigma_{\kk}^{2}=2\hs$.
This completes the proof of Theorem \ref{AsympH}.

\end{document}